\documentclass[12pt]{article}
\usepackage{amsmath,amsthm,amsfonts,amssymb, amscd,amssymb,latexsym}
\usepackage{graphicx}
\usepackage{hyperref}
\usepackage{srcltx}
\usepackage[usenames]{color}
\usepackage[matrix,arrow,curve]{xy}
\usepackage{textcomp}

\newtheorem{theorem}{Theorem}[section]
\newtheorem{lemma}[theorem]{Lemma}
\newtheorem{proposition}[theorem]{Proposition}
\newtheorem{statement}[theorem]{Statement}
\theoremstyle{definition}
\newtheorem{corollary}[theorem]{Corollary}
\newtheorem{definition}[theorem]{Definition}
\newtheorem{example}[theorem]{Example}
\newtheorem{problem}[theorem]{Problem}
\newtheorem{remark}[theorem]{Remark}

\renewcommand{\L}{{\mathtt{L}}}
\def\C{{\mathcal C}}
\def\A{{\mathcal{A}}}
\def\B{{\mathcal{B}}}

\def\N{{\mathcal{N}}}
\def\F{{\mathbb{F}}}
\def\E{{\mathcal{E}}}
\def\T{{\mathcal{T}}}

\newcommand{\EN}{{\mathbf{N}}}

\newcommand{\Th}{{\mathrm{Th}}}
\newcommand{\Tr}{{\mathrm{T}}}
\newcommand{\At}{{\mathrm{At}}}
\newcommand{\Mod}{{\mathrm{Mod}}}

\newcommand{\id}{{\mathrm{id}}}

\newcommand{\K}{{\mathbf{K}}}

\renewcommand{\S}{{\mathbf{S}}}
\renewcommand{\P}{{\mathbf{P}}}
\newcommand{\Pw}{{\mathbf{P_{\!\! \omega}}}}
\newcommand{\Ps}{{\mathbf{P_{\!\! s}}}}
\newcommand{\Pf}{{\mathbf{P_{\! f}}}}
\newcommand{\Pu}{{\mathbf{P_{\!\! u}}}}

\newcommand{\Ld}{{\underrightarrow{\mathbf{L}}}}
\newcommand{\Ls}{{\underrightarrow{\mathbf{L}}_{s}}}

\newcommand{\Lo}{{\mathbf{L_{fg}}}}
\newcommand{\e}{{\mathbf{\,_e}}}

\newcommand{\pvar}{{\mathbf{Pvar}}}

\newcommand{\qvar}{{\mathbf{Qvar}}}
\newcommand{\ucl}{{\mathbf{Ucl}}}
\newcommand{\Res}{{\mathbf{Res}}}
\newcommand{\Dis}{{\mathbf{Dis}}}

\newcommand{\ac}{{\mathrm{ac}}}

\newcommand{\Tf}{{\mathrm{T}}}
\newcommand{\Hom}{{\mathrm{Hom}}}
\newcommand{\Rad}{{\mathrm{Rad}}}
\newcommand{\V}{{\mathrm{V}}}
\newcommand{\Mor}{{\mathrm{Mor}}}

\newcommand{\CA}{{\mathbf{CA}}}
\newcommand{\AS}{{\mathbf{AS}}}

\title{Algebraic geometry over algebraic structures II: Foundations}

\author{E.\,Daniyarova, A.\,Myasnikov, V.\,Remeslennikov}

\begin{document}
\maketitle

\begin{abstract}
MSC 03C99; 08A99; 14A99

In this paper we introduce elements of algebraic geometry over an
arbitrary algebraic structure. We prove so-called Unification
Theorems which describe coordinate algebras of algebraic sets in
several different ways.

\end{abstract}

\tableofcontents

\section{Introduction}\label{sec:introduction}

Quite often relations between sets of elements of a fixed
algebraic structure $\A$ can be described in terms  of equations
over $\A$. In the classical case when $\A$ is a field the area of
mathematics that studies such relations is {\it algebraic geometry}.
It is therefore reasonable to use the same
name in the general case.

Algebraic geometry over algebraic structures is a new
area of research in modern algebra. Nevertheless, there are already several
breakthrough results for particular algebraic structures as well as
an interesting development of a general theory.

To date, the most developed branch of algebraic geometry over
algebraic structures is  algebraic geometry over groups. Most
notable is the solution of the main problem of algebraic
geometry~--- the classification of algebraic sets and coordinate
groups in the case of free groups. The classification of
coordinate groups is given in the language of free constructions
and is a result of joint effort of many mathematicians. The most
important papers in this direction are by
R.\,C.\,Lyndon~\cite{Lyndon}, K.\,I.\,Appel~\cite{Appel},
R.\,Bryant~\cite{Bryant}, G.\,Makanin~\cite{Makanin},
A.\,Razborov~\cite{Razborov1, Razborov2}, R.\,I.\,Grigorchuk and
P.\,F.\,Kurchanov~\cite{GrigorchukKurchanov},
Z.\,Sela~\cite{Sela1, Sela2, Sela3}, A.\,Myasnikov,
V.\,Re\-mes\-len\-ni\-kov, D.\,Serbin~\cite{MR1, MRS, Remesl}. The
final results were obtained in a break-through series of papers by
O.\,Khar\-lam\-po\-vich and A.\,Myas\-ni\-kov~\cite{KM1, KM2, KM3,
KM4}.

A significant progress was made in algebraic geometry over free
metabelian groups~\cite{Chapuis, Rem2, Rem-Shtor2,
Rem-Shtor1,Rem-Rom1, Rem-Rom2, Rem-Tim, Rom1}. The case of
solvable groups was considered in~\cite{GRom, MRom, Rom2}. In the
last few years considerable progress has been made towards
understanding algebraic geometry over partially commutative
groups. Here we would like to mention the following
papers~\cite{CasKaz1, CasKaz2, CasKaz3, GTim, Misch2, Tim}.

Algebraic geometry over algebraic structures is also being
developed for algebraic structures other than groups. Nice results
were obtained in algebraic geometry over commutative monoids with
cancellation~\cite{MorShev, Shevl2, Shevl3}. A certain progress is
achieved in algebraic geometry over non-associative algebras,
namely over Lie algebras~\cite{ChSh, DKR1, DKR2, Daniyarova2,
DKR3, Daniyarova3, Rem-Shtor3, RomSh}, and over anti-commutative
algebras~\cite{DaniyarovaOnskul}.

Note that there are lots of papers on solving particular equations over particular algebras.
In this short introduction, we do not pretend to account for all these papers,
 rather we only mention those papers that demonstrate the importance and necessity of algebraic geometry over algebraic structures.

The accumulated analysis of the structure of algebraic sets and
coordinate algebras over particular algebraic structures (groups,
monoids, rings, algebras, etc.) creates a need for a general
framework. From this perspective, there are general results which
hold when one studies algebraic geometry over an arbitrary
algebraic structure, we refer to such results, and, more
generally, to such a viewpoint, as to the {\it universal algebraic
geometry}. Research in this area has been initiated in a series of
papers by B.\,I.\,Plotkin~\cite{Plot1,Plot2,Plot3}, G.\,Baumslag,
O.\,G.\,Kharlampovich, A.\,G.\,Myasnikov and
V.\,N.\,Remeslennikov~\cite{BMR1,KM1,KM2,MR2}.

Universal algebraic geometry, firstly, is a transfer of general
notions and ideas from algebraic geometry over particular
algebraic structures to the case of an arbitrary algebraic
structure; secondly, it is formulation and proof of general
results without the use of properties of a concrete algebraic
structure; thirdly, it is development of a general theory with its
own naturally arising problems and goals. One can  point out
several papers with general results for particular algebraic
structures. Most of the results in these papers are proven using
techniques and properties specific for the structures considered.
Universal algebraic geometry presents standard and universal means
of proving those results using the framework of universal algebra
and model theory.

\bigskip

This paper is the second in our series of papers on universal
algebraic geometry. In the first paper of this
series,~\cite{DMR1}, we present the background material from
universal algebra and model theory as needed for universal
algebraic geometry and discuss how model-theoretic notions and
ideas work in universal algebraic geometry. As this paper is a
continuation of the authors' previous paper~\cite{DMR1}, we
suggest the reader to consult that paper prior to reading this
one. For the sake of convenience, and in an attempt to make the
paper more self-contained, we present some of the more essential
notations and definitions from~\cite{DMR1} (see
Section~\ref{sec:preliminaries}).

The main aim of our previous paper~\cite{DMR1} is to prove the
 so-called Unification Theorems (Theorem~A and Theorem~B) which
give a description of coordinate algebras of irreducible algebraic
sets from several different viewpoints. Let us note that,
following R.\,Hartshorne~\cite{Hartshorne}, in our papers all
irreducible algebraic sets are non-empty.

\medskip

\noindent {\bf Theorem~A.} {\it Let $\A$ be an equationally
Noetherian algebraic structure in $\L$. Then for a finitely
generated algebraic structure $\C$ of $\L$ the following
conditions are equivalent:
\begin{enumerate}
\item [1)] $\Th_{\forall} (\A) \subseteq \Th_{\forall} (\C)$, i.e., $\C \in \ucl(\A)$;
\item [2)] $\Th_{\exists} (\A) \supseteq \Th_{\exists} (\C)$;
\item [3)] $\C$ embeds into an ultrapower of $\A$;
\item [4)] $\C$ is discriminated by $\A$;
\item [5)] $\C$ is a limit algebraic structure over $\A$;
\item [6)] $\C$ is an algebraic structure defined by a complete atomic type
in the theory $\Th _{\forall} (\A)$ in $\L$;
\item [7)] $\C$ is the coordinate algebra of an
irreducible algebraic set over $\A$ defined by a system of
equations in the language $\L$.
\end{enumerate}
}

\medskip

We begin current paper with a detailed exposition of the
foundations of universal algebraic geometry. In
Section~\ref{sec:elements} we introduce the basic notions of
algebraic geometry over an arbitrary algebraic structure $\A$:
equation over $\A$, algebraic set over $\A$, radical, coordinate
algebra, the Zariski topology, the notions of irreducible sets and
equationally Noetherian algebras.

The main results of this paper are the following theorems.

\medskip

\noindent {\bf Theorem}~\ref{dual}{\bf.} {\it The category $\AS
(\A)$ of algebraic sets over an algebraic structure $\A$ and the
category $\CA (\A)$ of coordinate algebras of algebraic sets over
$\A$ are dually equivalent.}

\medskip

\noindent {\bf Theorem~C.} {\it Let $\A$ be an equationally
Noetherian algebraic structure in a language $\L$. Then for a
finitely generated algebraic structure $\C$ of $\L$ the following
conditions are equivalent:
\begin{enumerate}
\item [1)] $\C \in \qvar(\A)$, i.e., $\Th_{\rm qi} (\A) \subseteq \Th_{\rm qi} (\C)$;
\item [2)] $\C \in \pvar(\A)$;
\item [3)] $\C$ embeds into a direct power of $\A$;
\item [4)] $\C$ is separated by $\A$;
\item [5)] $\C$ is a subdirect product of finitely many limit algebraic structures over $\A$;
\item [6)] $\C$ is an algebraic structure defined by a complete atomic type in the theory $\Th _{\rm qi} (\A)$ in $\L$;
\item [7)] $\C$ is the coordinate algebra of an algebraic set over $\A$ defined by a system of equations in the language $\L$.
\end{enumerate}
}

Theorem~C continues a series of Unification Theorems in algebraic
geometry that we have begun in~\cite{DMR1}. Theorem~A gives a
description of coordinate algebras of irreducible algebraic sets.
In classical algebraic geometry over a field, irreducible
algebraic sets determine the whole picture. Unlike the classical
case, in algebraic geometry over an arbitrary algebraic structure
$\A$ it is not so and here we need a description of all algebraic
sets and all coordinate algebras. Here Theorem~C is helpful.

Let us note that items~5) in Theorems~A~and~C give a description
of coordinate algebraic structures via limit algebraic structures.
Limit algebraic structures (for the most part, groups) become the
object of an intense study in modern algebra~\cite{CG, GS,
Groves1, Groves1-2, Groves2-1, Groves2, G1}. The definitions of a
limit algebraic structure and of an algebraic structure defined by
a complete atomic type require a lot of preliminary material and
are omitted in this paper (see~\cite[Subsections~4.2
and~5.1]{DMR1}).

In the previous~\cite{DMR1} paper and this one we suppose that a
language $\L$ is functional, i.e., it has no predicates. This
restriction is not a fundamental matter: all of proved here
results stay true in the case of an arbitrary signature $\L$.
However, if $\L$ has predicates then definitions of all notions
that we introduce will become more complicated, the volume of the
paper will become bigger, and a reader will be expected more
grounding to understand the paper. We will describe the case of an
arbitrary signature in an addition to this paper.

\medskip

Summarising, in our work we set up the foundation of universal
algebraic geometry. The presented material can be considered as a
guide for studying algebraic geometry over particular algebraic
structures. In Section~\ref{sec:problems} of the paper we present
several open problems in algebraic geometry over free monoids,
free Lie algebras and free associative algebras. Before applying
universal algebraic geometry to a particular group, ring, monoid
etc., we suggest to draw attention to the following remark.

There are three different segments of algebraic geometry over a
particular algebraic structure:
\begin{itemize}
\item[(i)] Coefficient-free algebraic geometry;
\item[(ii)] Diophantine algebraic geometry;
\item[(iii)] Algebraic geometry with
coefficients in some algebraic structure $\A$ and solutions in
some extension $\A < \B$ (usually, in some saturated model).
\end{itemize}
While laying the foundations of algebraic geometry over groups in
the papers~\cite{BMR1,MR2}, the authors choose the universal way
for explanation a material: they talk about algebraic geometry
over a group $H$ with coefficients in a given group $G$, $G\leq
H$. Along these lines, the notions of a $G$-group,
$G$-homomorphism, $G$-formula, etc. arise naturally. Obviously,
this approach is useful for all the three above mentioned segments
of algebraic geometry over groups: for coefficient-free algebraic
geometry (set $G=1$), for Diophantine algebraic geometry (set
$G=H$). The situation is the same in algebraic geometry over Lie
algebras~\cite{Daniyarova1}, monoids, rings, and so on. Note that
for semigroups and any other algebra without the trivial
subalgebra such ``universal'' approach does not work. However,
universal algebraic geometry provides an instrument for analysing
the three of the above segments in a uniform way, using one
technique: it just suffices to choose the ground language $\L$
appropriately. For instance, when  studying coefficient-free
algebraic geometry over a semigroup $G$, one should choose the
language $\L=\{\cdot\}$. For Diophantine algebraic geometry over
$G$ it is only natural to take the extended language $\L_G$ as the
ground language (see Section~\ref{sec:preliminaries} for the
definition of $\L_G$). Similarly, for algebraic geometry over a
semigroup $H$ with coefficients in $G$, $G\leq H$, the signature
$\L_G$ also works well.

Mathematical logic, model theory and universal algebra are the
background of universal algebraic geometry. Hence, it is only
natural that the choice of the ground language $\L$ plays a
crucial role in universal algebraic geometry, as all definitions
that we give depend on the ground language $\L$. Since when we
talk about algebraic structures, formulas, theories we always
assume that a certain language is fixed, so no confusion arises.
Therefore, when one considers our definitions and results in the
context of particular algebraic structures (e.g. groups, monoids,
algebras, etc.) it is necessary to point out the language in which
this group (monoid, algebra, etc.) is considered.

\section{Preliminaries}\label{sec:preliminaries}

In this section we present basic notations from model theory that
we use in this paper. For more detailed information we refer
to~\cite{DMR1, Gorbunov, Marker}.

Let $\L$ be a first-order functional language, $X = \{x_1, x_2,
\ldots, x_n\}$ a finite set of variables, $\Tr_\L(X)$ the set of
all terms of $\L$ with variables in $X$, $\T_\L(X)$ the absolutely
free $\L$-algebra with basis $X$ and $\At_{\L}(X)$ the set of all
atomic formulas of $\L$ with variables in $X$.

Typically we denote algebraic structures in $\L$ by capital
calligraphic letters $\A, \B, \C,\ldots$ and their universes (the
underlying sets) by the corresponding capital Latin letters $A, B,
C, \ldots$. Algebraic structures in a functional language are
termed {\em algebras}.

In this paper we use some operators which image a class $\K$ of
$\L$-algebras into another one. For the sake on convenience we
collect here the list of all these operators:\\
$\S (\K)$~--- the class of subalgebras of algebras from $\K$;\\
$\P (\K)$~--- the class of direct products of algebras from
$\K$;\\ $\Pw (\K)$~--- the class of finite direct products of
algebras from $\K$;\\
$\Ps (\K)$~--- the class of subdirect products of algebras from
$\K$;\\
$\Pf (\K)$~--- the class of filterproducts of algebras from
$\K$;\\
$\Pu (\K)$~--- the class of ultraproducts of algebras from $\K$;\\
$\Ld (\K)$~--- the class of direct limits of algebras from $\K$;\\
$\Ls (\K)$~--- the class of epimorphic direct limits of algebras
from $\K$;\\
$\Lo(\K)$~--- the class of algebras in which all finitely
generated subalgebras belong to $\K$;\\
$\pvar (\K)$~--- the least prevariety including $\K$;\\
$\qvar (\K)$~--- the least quasi-variety including $\K$, i.e.,
$\qvar (\K)=\Mod (\Th _{\rm qi} (\K))$;\\
$\ucl (\K)$~--- the universal class of algebras generated by $\K$,
i.e., $\ucl (\K)=\Mod (\Th _{\forall} (\K))$; \\$\Res (\K)$~---
the class of algebras which are separated by $\K$;
\\$\Dis (\K)$~--- the class of algebras which are discriminated by
$\K$;
\\
$\K\e$~--- the addition of the trivial algebra $\E$ to $\K$, i.e., $\K\e=\K\cup \{\E\}$;\\
$\K_\omega$~--- the class of finitely generated algebras from
$\K$.

Here we denote by $\Th_{\rm qi}(\K)$ (correspondingly,
$\Th_{\forall}(\K)$, $\Th_{\exists}(\K)$) the set of all
quasi-identities (correspondingly, universal sentences,
existential sentences) which are true in all structures from $\K$.

For an arbitrary class $\K$ of $\L$-algebras one has:
$$
\Dis(\K)\;\; \subseteq \;\; \ucl(\K)=\S\Pu(\K) \;\; \subseteq \;\;
\qvar(\K),
$$
$$
\Dis(\K) \;\; \subseteq \;\; \Res(\K)=\S\P(\K)=\pvar(\K) \;\;
\subseteq \;\; \qvar(\K).
$$

According to Gorbunov~\cite{Gorbunov} and in contrast
to~\cite{DMR1}, we assume that the direct product for the empty
set of indexes coincides with the trivial $\L$-algebra $\E$. In
particularly, $\E \in \P(\K)$, $\E \in \Pw(\K)$, $\E \in \Ps(\K)$,
for an arbitrary class of $\L$-algebras $\K$. However, while
defining an filterproduct we assume that the set of indexes is
non-empty.

Let us remind the definitions of separation and discrimination.

\begin{definition}
An $\L$-algebra $\C$ is {\it separated} by a class of
$\L$-algebras $\K$ if for any pair of non-equal elements $c_1, c_2
\in C$ there is a homomorphism $h \colon \C \to \B$ for some
$\B\in \K$, such that $h(c_1)\ne h(c_2)$.
\end{definition}

\begin{definition}
An $\L$-algebra $\C$ is {\em discriminated} by $\K$ if for any
finite set $W$ of elements from $C$ there is a homomorphism $h
\colon\C \to \B$ for some $\B\in \K$, whose restriction onto $W$
is injective.
\end{definition}

If $\C$ is separated (discriminated) by a class $\K=\{\B\}$, then
we say that $\C$ is separated (discriminated) by algebra $\B$. Let
us note that in the definitions above we do not claim a
homomorphism $h$ is an epimorphism. As it follows from the
definitions of separation and discrimination, the trivial algebra
$\E$ is separated by a class $\K$ anyway, and $\E$ is
discriminated by $\K$ if and only if there exists a homomorphism
$h \colon\E \to \B$ for some $\B\in \K$, i.e., $\B$ has a trivial
subalgebra. This circumstance allows to reserve inclusion $\Dis
(\K) \subseteq \ucl (\K)$ and identity $\Res (\K)=\S\P(\K)$ in a
habitual form.

\medskip

Unification Theorem~A has been proven in~\cite{DMR1} for
equationally Noetherian algebras. However, the following result
holds in more general case too.

\begin{proposition}[\cite{DMR1}]\label{logirr}
Let $\A$ be an algebra in a language $\L$. Then for a finitely
generated $\L$-algebra $\C$ the following conditions are
equivalent:
\begin{itemize}
\item $\Th_{\forall} (\A) \subseteq \Th_{\forall} (\C)$, i.e., $\C \in \ucl(\A)$;
\item $\Th_{\exists} (\A) \supseteq \Th_{\exists} (\C)$;
\item $\C$ embeds into an ultrapower of $\A$;
\item $\C$ is a limit algebra over $\A$;
\item $\C$ is an algebra defined by a complete atomic type in the theory $\Th _{\forall} (\A)$ in
$\L$.
\end{itemize}
\end{proposition}

\medskip

For $\L$-algebra $\A$ we denote by $\L_\A = \L \cup\{c_a \mid a
\in A\}$ the language $\L$ extended by elements from $\A$ as new
constant symbols.

An algebra $\B$ in the language $\L_\A$ is called $\A$-algebra if
$\B \models {\rm Diag}(\A)$. It means that $\A$ embeds into $\B$
and the corresponding embedding $\lambda\colon \A \to \B$ is
fixed. Remind that the diagram ${\rm Diag}(\A)$ of $\A$ is the set
of all atomic sentences from $\At_{\L_\A}(\emptyset)$ or their
negations which are true in $\A$.

Let $\B$ and $\C$ be $\A$-algebras, and $h\colon \B \to \C$ a
$\L_\A$-homomorphism. We usually refer to $h$ as to
$\A$-homomorphism. Similarly, we define $\A$-separation and
$\A$-discrimination. The prevariety of $\B$ in the language
$\L_\A$ we denote by $\pvar_\A(\B)$, the quasivariety~--- by
$\qvar_\A(\B)$, and the universal closure~--- by $\ucl_\A(\B)$.
Such notations is especially convenient when $\B=\A$. In this case
we have to distinguish $\A$ as $\L$-algebra and $\A$ as
$\L_\A$-algebra. Correspondingly, for instance, we have to point
out what class we concerned with: $\ucl (\A)$ or $\ucl_\A(\A)$.

\section{Elements of algebraic geometry}\label{sec:elements}

Let $\L$ be a functional language and $\A$ an $\L$-algebra.

In this section we introduce the basic notions of universal
algebraic geometry: equation in the language $\L$, algebraic set
over the algebra $\A$, radical, coordinate algebra, the Zariski
topology, irreducible set, equationally Noetherian algebra.

\subsection{Equations and algebraic sets} \label{subsec:alg_sets}

Let $X=\{x_1,\ldots,x_n\}$ be a finite set of variables.

\begin{definition}
Atomic formulas from $\At_{\L}(X)$ are called {\em equations} in
$\L$ with variables in $X$. Any subset $S\subseteq \At_{\L}(X)$ is
called a {\em system of equations} in $\L$.
\end{definition}

Sometimes, to emphasize that formulas are from $\L$ we call such
equations (and systems of equations) {\em coefficient-free
equations}, meanwhile, in the case  when $\L = \L_\A$, we refer to
such equations as {\em equations with coefficients in algebra
$\A$} or $\A$-{\em equations}.

When someone looks for solutions of equations and systems of
equations in algebra $\A$ it is said to be {\em algebraic geometry
over the algebra $\A$}. Algebraic geometry over an algebra $\A$ in
the language $\L_\A$ is called {\em Diophantine}. If $\B$ is an
$\A$-algebra then investigation into algebraic geometry over $\B$
as over $\L_\A$-algebra is called {\em algebraic geometry over
$\B$ with coefficients in $\A$}.

We term the set
$$
A^n = \{(a_1,\ldots,a_n) \mid a_i \in A\}
$$
{\em affine $n$-space} over algebra $\A$ and we sometimes refer to
its elements as {\em points}. A point $p=(a_1,\ldots,a_n)\in A^n$
is called a {\em root} of equation $(t_1=t_2)$, $t_1,t_2 \in
\Tr_\L (X)$, if $\A \models (t_1=t_2)$ via interpretation $x_i
\mapsto a_i$, $i=\overline{1,n}$. Further, a point $p$ is a root
of system of equations $S\subseteq \At_{\L}(X)$ if it is a root of
every equation from $S$.

\begin{definition}
Let $S$ be a system of equations in the language $\L$ in variables
$X$. The set of all roots of the system $S$ in the affine
$n$-space $A^n$ we denote by $\V _{\A}(S)$ (or briefly $\V (S)$):
$$
\V_{\A}(S)\; =\; \{\,(a_1,\ldots,a_n)\in A^n \; \vert \;
t^{\A}_1(a_1,\ldots,a_n)=t^{\A}_2(a_1,\ldots,a_n) \quad \forall \;
(t_1=t_2)\in S \,\}
$$
The set $\V _{\A}(S)$ is called {\em algebraic set} over the
algebra $\A$ defined by the system $S$.
\end{definition}

A system $S$ is called {\em inconsistent} over $\A$ if $\V _{\A}
(S) = \emptyset$, otherwise it is called {\em consistent}. We say
two systems of equations $S_1$ and $S_2$ are {\em equivalent} over
$\A$ and write $S_1 \sim _\A S_2$ if $\V _{\A} (S_1) = \V _{\A}
(S_2)$.

\begin{example}\label{ex:sing}
Any points, affine $n$-spaces, and direct products of algebraic
sets give standard examples of algebraic sets.
\begin{enumerate}
\item In Diophantine algebraic geometry over an algebra $\A$, every
point from the affine $n$-space $A^n$ is an algebraic set. For a
point $(a_1,\ldots,a_n) \in A^n$ one has
$$
S = \{x_1 = c_{a_1}, \ldots, x_n = c_{a_n}\}, \quad \V  (S) =
\{(a_1, \ldots, a_n)\}.
$$
\item The affine $n$-space $A^n$ is an algebraic set for the degenerate
system $S = \{x=x\}$.
\item Let $Y\subseteq A^n$ and $Z\subseteq A^m$ be algebraic sets
over $\A$. Then $Y\times Z \subseteq A^{n+m}$ is an algebraic over
$\A$ too. Indeed, if $Y=\V(S)$, $S\subseteq
\At_\L(x_1,\ldots,x_n)$, and $Z=\V(S')$, $S'\subseteq
\At_\L(x'_1,\ldots,x'_m)$, then
$$
Y\times Z=\V (S\cup S'), \quad S\cup S'\subseteq
\At_\L(x_1,\ldots,x_n,x'_1,\ldots,x'_m).
$$
\end{enumerate}
\end{example}

Now we show some examples of algebraic sets over specific
algebras: free group, free Lie algebra, min-max structure.

\begin{example}
Let us consider a free algebra $F$ with a free base
$a_1,\ldots,a_n$ in some variety $\Theta$ and the following
equation over $F$:
\begin{equation}\label{eq:w}
w(x_1,\ldots,x_n)=w(a_1,\ldots,a_n),
\end{equation}
where $w(a_1,\ldots,a_n)$ is an element in $F$. Then algebraic set
for equation~\eqref{eq:w} is
$$
Y=\{(\varphi (a_1),\ldots,\varphi(a_n))\: \mid \; \varphi \in
\mathrm{St} (w)\},
$$
wherehere $\mathrm{St} (w)=\{\varphi \in \mathrm{End} (F)\mid
\varphi (w)=w\}$. So, the algebraic set $Y$ is parametrized by
endomorphisms from the stabilizer $\mathrm{St} (w)$. Remind that
an element $w(a_1,\ldots,a_n)\in F$ is termed testing if
$\mathrm{St} (w) \subseteq \mathrm{Aut} (F)$. Thus, if $w$ is a
testing element then the algebraic set $Y$ has a parametrization
by means of automorphisms.

If $F$ is a finitely generated free group then there exists an
effective algorithm to determine $\mathrm{St} (w)$ for any
(cyclic) testing element $w\in F$~\cite{McCool}. Therefore, there
exists an effective algorithm for finding all solutions of
equation~\eqref{eq:w}.

In particular, the stabilizator $\mathrm{St} (w)$ of the
commutator $w=[a_1,a_2]$ in $F=\langle a_1, a_2 \rangle$ is
isomorphic to the free group of rank $2$. Indeed, $\mathrm{St} (w)
= \langle tr_1, tr_2\rangle$, where $tr_1$ is the automorphism of
$F$ defined by $a_1 \to a_2 a_1$, $a_2\to a_2$, and $tr_2$ is
defined by $a_1\to a_1$, $a_2 \to a_1 a_2$.
\end{example}

\begin{example}\label{Lie}
Let $L$ be a free Lie algebra of a finite rank over a field $k$.
An algebraic set $Y$ is called {\em bounded} if it enters into
some finite dimensional subspace of $L^n$ as $k$-linear space.
Bounded algebraic sets over $L$ have been classified
in~\cite{Daniyarova3}. Any finite dimensional subspace in $L$
supplies an elemental example of bounded algebraic set. For linear
subspace $W$ in $L$ with basis $v_1,\ldots,v_m$ we have:
$$
s_1(x)=[x,v_1], \; s_2(x)=[[x,v_1],[v_2,v_1]], \; \ldots, \;
s_m(x)=[s_{m-1}(x),s_{m-1}(v_m)],
$$
$$
\V (s_m)=W.
$$
While $W$ is an algebraic set in one variable $x$, the similar
algebraic sets in $n$ variables are called $n$-parallelepipeds.
Under $n$-parallelepiped $\mathbf{W}$ we mean a Cartesian product
of an $n$-tuple of finite dimensional subspaces $W_1,\ldots, W_n$
in $L$: $ \mathbf{W}=W_1 \times \ldots \times W_n$. The dimension
of $n$-parallelepiped $\mathbf{W}$ is defined by $\mathrm{dim}
(\mathbf{W})=\mathrm{dim} (W_1)+\ldots+\mathrm{dim}(W_n)$. An
algebraic set $Y\subseteq L^n$ is bounded by parallelepiped
$\mathbf{W}$ if $Y\subseteq\mathbf{W}$.

\medskip

\noindent {\bf Theorem}~\cite{Daniyarova3}{\bf.} {\em Let
$\mathbf{W}$ be an $n$-parallelepiped over the free Lie algebra
$L$ over a field $k$. Algebraic sets over the algebra $L$ bounded
by the parallelepiped $\mathbf{W}$ lie in one-to-one
correspondence with algebraic sets over the field $k$ defined by
systems of equations in $\mathrm{dim} (\mathbf{W})$ variables.}

In~\cite{Rem-Shtor3} it has been shown that the equation
$[x,a]+[y,b]=0$ ($a,b \in L$, $[a, b]\ne 0$) has a complicated
solution over $L$ and its algebraic set is not bounded. However,
the same equation is easy to solve over a free anti-commutative
algebra $A$; and its algebraic set over $A$ is
bounded~\cite{DaniyarovaOnskul}.
\end{example}

\begin{example}\label{ex:minmax}
The algebraic structure $\mathcal{M}_\mathbb{R}=\langle
\mathbb{R}; {\max}, {\min}, {\cdot}, \allowbreak {+}, {-}, 0,
1\rangle$ with obvious interpretation of the symbols from
signature on $\mathbb{R}$ is an example of so-called min-max
structure.

\medskip

\noindent {\bf Theorem}~\cite{DvorKot}{\bf.} {\em A set
$Y\subseteq \mathbb{R}^n$ is algebraic over
$\mathcal{M}_\mathbb{R}$ if and only if it is closed in the
topology induced by the Euclidean metric on $\mathbb{R}^n$.}
\end{example}

The following two lemmas will be helpful further; they shows how
equivalence between systems of equations retains when we move from
algebra $\A$ to its direct and filter- powers.

\begin{lemma}
Let $\A$ be an $\L$-algebra and $\C$ a subalgebra of some direct
power of $\A$. Then for two system of equations
$S_1,S_2\subseteq\At_\L(X)$ condition $S_1\sim_\A S_2$ implies
$S_1\sim_\C S_2$.
\end{lemma}

\begin{proof}
Let $\C$ be a subalgebra of $\prod_{i\in I}{\A^{(i)}}$. For a
point $p=(c_1,\ldots,c_n)\in C^n$ let us write
$c_j=(a_j^{(i)})_{i\in I}$, $a_j^{(i)}\in A$, $j=1,\ldots,n$. We
have $p\in \V_\C(S_1)$ if and only if
$(a^{(i)}_1,\ldots,b^{(i)}_n)\in \V_\A (S_1)$ for every $i\in I$.
Since $S_1\sim_\A S_2$ the latter is equal to
$(a^{(i)}_1,\ldots,a^{(i)}_n)\in \V_\A (S_2)$ for every $i\in I$.
Therefore, $p\in \V_\C(S_1)$ if and only if $p\in \V_\C(S_2)$,
i.e., $S_1\sim_\C S_2$.
\end{proof}

\begin{lemma}\label{up}
Let $\A$ be an $\L$-algebra and $\C$ a subalgebra of some
filterpower of $\A$. Then for a system of equations
$S\subseteq\At_\L(X)$ and a finite subsystem $S_0 \subseteq S$
condition $S\sim_\A S_0$ implies $S\sim_\C S_0$.
\end{lemma}

\begin{proof}
Suppose $S\sim_\A S_0$ and $\C$ is a subalgebra of $\prod_{i\in I}
{\A ^{(i)}}/D$, where $D$ is a filter on $I$. Since inclusion
$\V_\C(S_0)\supseteq \V_\C(S)$ is obvious we need to prove only
the converse inclusion $\V_\C(S_0) \subseteq \V_\C(S)$.

Assume that a point $p=(c_1,\ldots,c_n)\in C^n$ is a root of
$S_0$. We write $c_j=(a_j^{(i)})_{i\in I}/D$, $j=\overline{1,n}$,
and denote $p^{(i)}=(a_1^{(i)},\ldots,a_n^{(i)})\in A^n$. For each
equitation $(t=s)\in S_0$ there exists an element $J$ in $D$ such
that $t(p^{(i)})=s(p^{(i)})$ for all $i\in J$. Since $S_0$ is
finite there is an element $J_0 \in D$ such that
$t(p^{(i)})=s(p^{(i)})$ for each equation $(t=s)\in S_0$ and each
index $i \in J_0$. As $\V_\A(S_0)=\V_\A(S)$ we have
$t(p^{(i)})=s(p^{(i)})$ for all $i \in J_0$ and every equation
$(t=s)\in S$. Hence, the point $p$ is a solution of $S$. Thus,
inclusion $\V_\C(S_0) \subseteq \V_\C(S)$ holds.
\end{proof}

\subsection{Radicals} \label{subsec:radicals}

With every algebraic set $Y$ we associate two important
objects~--- its radical $\Rad(Y)$ and coordinate algebra $\Gamma
(Y)$. In this subsection we will discuss radicals, and in the next
one we will discuss coordinate algebras.

\begin{definition}
For a subset $Y\subseteq A^n$ we term the following set of atomic
formulas from $\At_{\L} (x_1,\ldots,x_n)$ the {\em radical of the
set $Y$}:
$$
\Rad (Y) \; = \; \{\, (t_1=t_2) \; \vert \;
t^{\A}_1(a_1,\ldots,a_n)=t^{\A}_2(a_1,\ldots,a_n) \quad \forall \;
(a_1,\ldots,a_n)\in Y \, \}.
$$
\end{definition}

The radical $\Rad (Y)$ of an algebraic set $Y$ uniquely defines
it, i.e., an algebraic set $Y_1$ coincides with an algebraic set
$Y_2$ if and only if $\Rad (Y_1)=\Rad (Y_2)$ (see Lemma~\ref{rad}
below).

The {\em radical of a system of equations $S\subseteq \At_\L(X)$}
over an $\L$-algebra $\A$ is the set $\Rad (\V_\A (S))$. We denote
it $\Rad _{\A}(S)$ (or briefly $\Rad (S)$). Atomic formulas from
$\Rad _\A (S)$ are called {\em consequences} of the system $S$
over $\A$. So, atomic formula $(t_1=t_2)$ is a consequence of $S$
over $\A$ if and only if $\V(S) \subseteq \V(\{t_1=t_2\})$, i.e.,
$S \cup \{t_1=t_2\}\sim_{\A}S$. In other words, $\Rad (S)$ is the
maximal system of equations which is equivalent to $S$. Radical of
an inconsistent system $S$ coincides with $\At_{\L}(X)$.

By $[S]$ we denote the congruent closure of $S$, i.e., the least
congruent subset of $\At_{\L}(X)$, containing $S$~\cite{DMR1}. In
is obvious that $[S]\subseteq \Rad(S)$.

\begin{definition}
We name a subset $S\subseteq \At_{\L}(X)$ by {\em radical ideal}
over $\A$ if $S=\Rad_\A(Y)$ for some $Y\subseteq A^n$.
\end{definition}

One can consider $\Rad$ and $\V$ as operators. Thus, $\Rad$ is the
operator of calculation of radical for sets $Y\subseteq A^n$, and
$\V$ is the operator of calculation of algebraic sets for systems
$S \subseteq \At_{\L} (X)$. In the next lemma we gather elementary
properties of these operators.

\begin{lemma}\label{rad}
The following holds:
\begin{enumerate}
\item A subset $Y\subseteq A^n$ is
algebraic over $\A$ if and only if
$$
Y=\V_\A (\Rad (Y)).
$$
\item A subset $S \subseteq \At_{\L} (X)$ is a radical ideal over $\A$ if and only if
$$
S=\Rad_\A (S).
$$
\item\label{rad1} For any sets $Y_1,Y_2\subseteq A^n$ one has
$$
Y_1 \subseteq Y_2 \quad \Longrightarrow \quad \Rad (Y_1)\supseteq
\Rad (Y_2).
$$
\item\label{rad2} For any systems of equations $S_1,S_2 \subseteq \At_{\L} (X)$
one has
$$
S_1 \subseteq S_2 \quad \Longrightarrow \quad \V (S_1)\supseteq \V
(S_2)\quad \Longrightarrow \quad \Rad (S_1)\subseteq \Rad (S_2).
$$
\item\label{rad5} For any algebraic sets $Y_1,Y_2 \subseteq A^n$ one has
$$
Y_1 = Y_2 \quad \Longleftrightarrow \quad \Rad (Y_1)=\Rad (Y_2).
$$
\item\label{rad3} For any family of subsets $\{Y_i \subseteq A^n, i\in I \}$ one has
$$
\Rad (\:\bigcup\limits_{i\in I} {Y_i}\,) \; = \;
\bigcap\limits_{i\in I} {\Rad (Y_i)}.
$$
\item\label{rad4} For any family of systems of equations $\{S_i\subseteq \At_{\L} (X), i\in
I\}$ one has
$$
\V(\:\bigcup\limits_{i\in I} {S_i}\,) \; = \; \bigcap\limits_{i\in
I} {\V (S_i)}.
$$
\end{enumerate}
In particular, intersection of any family of algebraic sets in
$A^n$ is an algebraic set.
\end{lemma}

\begin{proof}
Straightforward.
\end{proof}

Lemma~\ref{rad} gives the following method for calculation the
radical $\Rad (Y)$ of an arbitrary non-empty set $Y \subseteq
A^n$. With a point $p=(a_1,\ldots,a_n) \in A^n$ we associate the
homomorphism $h_p\colon \T_{\L}(X)\to \A$ defined by
$h_p(t)=t^{\A}(a_1,\ldots,a_n)$. Clearly,
$$
t_1\sim_{\ker h_p} t_2 \quad \Longleftrightarrow \quad
(t_1=t_2)\in \Rad (\{p\}).
$$

\begin{lemma}
Let $Y$ be a non-empty algebraic set over an algebra $\A$. Then
\begin{equation}\label{eq:rad}
\theta_{\Rad (Y)}= \bigcap\limits_{p\in Y} {\ker h_p}.
\end{equation}
\end{lemma}

\begin{proof}
Indeed, by Lemma~\ref{rad}, $\Rad(Y)=\bigcap_{p\in Y} {\Rad
(\{p\})}$.
\end{proof}

\subsection{Coordinate algebras} \label{subsec:coo}

Let $S\subseteq \At_{\L}(X)$ be a system of equations and
$Y=\V_{\A}(S)$. It is not hard to see that the radical $\Rad (Y)$
is a congruent set of atomic formulas. Hence, it defines a
congruence on the absolutely free $\L$-algebra $\T_{\L}(X)$ that
we denote by $\theta_{\Rad (Y)}$~\cite{DMR1}:
$$
t_1 \sim _{\theta_{\Rad (Y)}} t_2 \quad \Longleftrightarrow \quad
(t_1=t_2)\in \Rad (Y), \quad t_1, t_2 \in \Tr_\L(X).
$$

\begin{definition}
The $\L$-structure
$$
\Gamma (Y) \; = \; \T_{\L}(X) \: / \: \theta_{\Rad (Y)}
$$
is called the {\em coordinate algebra} of the algebraic set $Y$.
\end{definition}

When $Y=\V_\A (S)$ we also refer to $\Gamma(Y)$ as to the {\em
coordinate algebra of the system $S$} over $\A$ and write
$\Gamma_\A (S)$ (or $\Gamma (S)$).

If $S\sim_\A\At_\L(X)$ then $\Gamma (S)$ is the trivial algebra
$\E$. For instance, one has $S\sim_\A\At_\L(X)$ if $S$ is
inconsistent over $\A$.

\begin{definition}
We say that an $\L$-algebra $\C$ is a {\it coordinate algebra over
an $\L$-algebra $\A$} if $\C \cong \Gamma (Y)$ for some algebraic
set $Y$ over $\A$.
\end{definition}

One of the principal goals of algebraic geometry over an algebra
$\A$ is to describe algebraic sets over $\A$ up to isomorphism
(the definition of isomorphism of algebraic sets see in
Subsection~\ref{subsec:categories} below). We will show that this
problem has two equivalent forms: the problem of classification of
radicals and the problem of classification of coordinate algebras
over $\A$.

While every algebraic set may be restored in unique manner from
its radical, it may be restored from its coordinate algebra just
up to isomorphism. The following result gives a specification of
algebraic sets by means of sets of homomorphisms. It shows how one
can restore an algebraic set from its coordinate algebra.

\begin{lemma}\label{hom}
Let $Y$ be a non-empty algebraic set over an $\L$-algebra $\A$.
Then points of $Y$ lie in one-to-one correspondence with
$\L$-homomorphisms from $\Hom (\Gamma (Y),\A)$.
\end{lemma}

\begin{proof}
Indeed, every homomorphism $h\colon \T_{\L}(X) / \theta_{\Rad (Y)}
\to \A$ is uniquely defined by the images of elements
$x_i/\theta_{\Rad (Y)}$, $i=\overline{1,n}$, i.e., by a point
$(a_1,\ldots,a_n)\in A^n$ with
$t^{\A}_1(a_1,\ldots,a_n)=t^{\A}_2(a_1,\ldots,a_n)$ for all
$(t_1=t_2)\in \Rad(Y)$. Clearly, the set of all appropriate points
$(a_1,\ldots,a_n)$ coincides with $Y$.
\end{proof}

\begin{corollary}
Points of a non-empty algebraic set $Y$ over an $\A$-algebra $\B$
lie in one-to-one correspondence with $\A$-homomorphisms from
$\Hom _{\A}(\Gamma (Y),\B)$.
\end{corollary}

\bigskip

In the classical algebraic geometry over a field one can consider
the coordinate ring as the ring of polynomial functions. Let us
discuss a similar idea in universal algebraic geometry.

\begin{definition}
For a non-empty set $Y\subseteq A^n$ and a term $t\in \Tr_{\L}(X)$
we refer to the map $t^Y\colon Y\to A$ defined by
$$
t^Y(p)=t^{\A}(a_1,\ldots,a_n), \quad p=(a_1,\ldots,a_n), \quad
p\in Y,
$$
as a {\em term function} on $Y$. We call the set $\Tf (Y)$ of all
term functions on $Y$ with obvious interpretation of signature
symbols from $\L$ the {\em algebra of term functions on $Y$}.
\end{definition}

\begin{lemma}\label{termfunc}
For a non-empty algebraic set $Y$ over an $\L$-algebra $\A$ one
has
$$
\Gamma (Y) \cong \Tf (Y).
$$
\end{lemma}

\begin{proof}
Let $h\colon \T_{\L}(X) \to \Tf(Y)$ be the epimorphism defined by
$h(t)=t^Y$, $t\in \Tr_\L(X)$. One has $\T_{\L}(X) /\ker h \simeq
\Tf (Y)$. On the other hand, $t_1\sim_{\ker h} t_2$ if and only if
$(t_1=t_2) \in \Rad (Y)$, $t_1,t_2 \in \Tr_{\L} (X)$. Therefore,
$\T_{\L}(X) /\ker h \simeq \Gamma (Y)$.
\end{proof}

\begin{example}
Let $Y=\{(a_1,\ldots,a_n)\}$ be a singleton algebraic set from
Example~\ref{ex:sing}. Then coordinate algebra $\Gamma (Y)$ is
$\A$-isomorphic to the algebra $\A$. Indeed, it is easy to see
that $\Tf(Y)\cong_\A \A$.
\end{example}

\bigskip

The empty set $\emptyset$ is not necessary an algebraic set over
an algebra $\A$.

\begin{example}
Let $\L=\{\cdot, ^{-1}, e\}$ be the language of groups and $G$ a
group. Every equation $t(x_1,\ldots,x_n)=s(x_1,\ldots,x_n)$ in
$\L$ has at least one root in $G$, namely, $x_1=\ldots=x_n=e$.
Thus, the empty set is not algebraic over $\langle G ; \L
\rangle$.
\end{example}

On the other hand, if $\L$ is a language containing at least two
constant symbols $c_1, c_2$, and $\A$ an $\L$-algebra with
$c_1^\A\ne c_2^\A$, then the empty set is algebraic over $\A$,
since $\V_\A (\{c_1=c_2\})=\emptyset$.

\begin{lemma}[on the empty set and the trivial algebra]\label{emptyset}
For an $\L$-algebra $\A$ the following hold:
\begin{itemize}
\item[1)] The empty set is an algebraic over $\A$ if and only if $\A$
has not a trivial subalgebra.
\item[2)] If the empty set
is algebraic over $\A$ then $\V_\A (\At_{\L}(X))=\emptyset$ for
every finite set $X$.
\item[3)] The trivial algebra $\E$ is a coordinate algebra over $\A$
anyway. Moreover, if $Y$ is an algebraic set over $\A$ such that
$\E=\Gamma (Y)$ then $Y$ is either irreducible or $Y=\emptyset$
(the definition of irreducible algebraic set see in
Subsection~\ref{subsec:topology} below).
\item[4)] The trivial algebra $\E$ is the coordinate algebra of an
irreducible algebraic set over $\A$ if and only if $\A$ has a
trivial subalgebra.
\end{itemize}
\end{lemma}

\begin{proof}
Suppose that $\A$ has a trivial subalgebra $\E=\langle \{e\};
\L\rangle$. Then for every term $t\in \Tr_\L(X)$ one has
$t(e,\ldots,e)=e$. Thus, every system of equation $S$ has a root
$p=(e,\ldots,e)$, and the empty set is not algebraic over $\A$.
Assume now that $\emptyset$ is not algebraic over $\A$. Then there
exists an element $e\in A$ such that $t(e)=e$ for all terms $t\in
\Tr_\L(\{x\})$. It is clear that the element $e$ forms a trivial
subalgebra of $\A$. We have just proven item~1). To show~2) assume
that $\emptyset$ is an algebraic set over $\A$. Then there exists
a natural number $n$ and an inconsistent over $\A$ system of
equations $S(x_1,\ldots,x_n)$. Thus, $S'(x)=S(x,\ldots,x)$ is
inconsistent too. Hence, for every finite set $X$ one has the
inconsistent system $S'\subseteq \At_\L(X)$, thus $\V_\A
(\At_{\L}(X))=\emptyset$.

The first statement in item~3) is obvious as far as $\E=\Gamma
(\At_\L(X))$. The second one will be proven in
Corollary~\ref{cor14} bellow. To prove item~4) let us assume that
$\A$ has a trivial subalgebra, then, according to item~1), the
empty set is not an algebraic set over $\A$. By item~3), it means
that the trivial algebra is the coordinate algebra of an
irreducible algebraic set over $\A$. Conversely, suppose that $\A$
has no a trivial subalgebra, then the empty set is algebraic over
$\A$. If even so $\E$ is the coordinate algebra of an irreducible
algebraic set $Y$ over $\A$, $Y=\V(S)$, $S \subseteq \At_\L(X)$,
then $\Rad (Y)=\At_\L(X)$. However, according to item~2), in this
case $\V_\A (\At_{\L}(X))=\emptyset$, that means $Y=\emptyset$.
The letter is contradicts to the definition of irreducible set.
\end{proof}

\begin{remark}
If an $\L$-algebra $\A$ has a trivial subalgebra then there exists
an element $e\in A$ such that $c^\A=e$ for all constant symbols
$c\in\L$. Suppose we study Diophantine algebraic geometry over a
non-trivial group $G$. Then the trivial subgroup $1$ of $G$ is not
a trivial subalgebra of $G$ in terms of model theory. As the
ground language here is $\L_G$, therefore, the trivial subgroup
$1$ is not $\L_G$-substructure of $G$ at all.
\end{remark}

\bigskip

The following proposition and its corollaries are helpful for the
problem of classifying of coordinate algebras over $\A$.

\begin{proposition}\label{coo}
Let $\A$ be an algebra in $\L$. Then for a finitely generated
algebra $\C$ of $\L$ the following conditions are equivalent:
\begin{itemize}
\item[1)] $\C \in \pvar (\A)$ ;
\item[2)] $\C$ embeds into a direct power of $\A$;
\item[3)] $\C$ is separated by $\A$;
\item[4)] $\C$ is the coordinate algebra of an algebraic set over $\A$ defined by a system of equations in $\L$.
\end{itemize}
\end{proposition}

\begin{proof}
Equivalence $1) \Longleftrightarrow 2) \Longleftrightarrow 3)$ has
been proven in~\cite[Lemma~3.5]{DMR1} in the form $\pvar(\A)=\S\P
(\A)=\Res(\A)$.

Suppose $\C$ is the coordinate algebra of an algebraic set $Y$
over $\A$. If $Y=\emptyset$ then $\Gamma (Y)=\E$, and inclusion
$\C \in \pvar(\A)$ is evident. So, we assume that $Y$ is
non-empty. The equality~\eqref{eq:rad} induces the monomorphism
$h\colon\T_{\L}(X)/\theta_{\Rad(Y)} \to \prod_{p\in
Y}{\T_\L(X)/\ker h_p}$~\cite[Lemma~3.1]{DMR1}. Since
$\T_\L(X)/\ker h_p \cong \mathrm{Im}\, h_p$ is a subalgebra of
$\A$ we have the embedding $h\colon\T_{\L}(X)/\theta_{\Rad(Y)} \to
\A^{|Y|}$. Thus, we have proved implication $4) \Longrightarrow
2)$.

Let us show implication $3) \Longrightarrow 4)$. Suppose that $\C$
is a finitely generated $\L$-algebra from $\Res (\A)$ with a
finite generating set $X=\{x_1, \ldots, x_n\}$. If $\C$ is the
trivial algebra then there is nothing to prove. Assume that $\C$
is non-trivial, and $\C = \langle X \mid S\rangle$ is a
presentation of $\C$ in the generators $X$, where $S \subset
\At_{\L}(X)$. The latter means that $\C \cong
\T_{\L}(X)/\theta_S$. It is sufficient to show that
$\C=\Gamma(S)$, i.e., $\Rad _{\A}(S)=[S]$. Since $\C$ is separated
by $\A$ for any atomic formula $(t=s)\not\in [S]$ there exists a
homomorphism $h\colon\C \to \A$ with
$t^{\A}(h(x_1),\ldots,h(x_n))\ne s^{\A}(h(x_1),\ldots,h(x_n))$.
Obviously, $(h(x_1),\ldots,h(x_n)) \in \V _{\A}(S)$, so
$(t=s)\not\in \Rad _{\A}(S)$. It proves that $\Rad _{\A}(S)=[S]$.
\end{proof}

\begin{corollary} \label{cor4}
Let $\C$ be the coordinate algebra of an algebraic set over an
algebra $\A$ and $\langle X \mid S\rangle$ a presentation of $\C$
in the generators $X$ with $S \subseteq \At_{\L}(X)$. Then $[S]$
is a radical ideal over $\A$.
\end{corollary}

\begin{proof}
Since $\C$ is separated by $\A$ we may repeat the arguments above.
\end{proof}

\begin{corollary} \label{cor1}
The class of all coordinate algebras of algebraic sets over an
algebra $\A$ coincides with $\pvar (\A)_\omega$.
\end{corollary}

\begin{corollary} \label{cor8}
Let a finitely generated $\L$-algebra $\C$ imbeds into a direct
product of coordinate algebras of some algebraic sets over $\A$.
Then $\C$ is the coordinate algebra of an algebraic set over $\A$.
\end{corollary}

\begin{corollary} \label{cor2}
For any algebraic set $Y$ over algebra $\A$ one has $\Gamma (Y)
\in \qvar (\A)$. In particular, $\Gamma (Y)$ satisfies all
identities and quasi-identities in $\L$ which hold in $\A$.
\end{corollary}

\begin{proof}
It follows from inclusion $\pvar(\A) \subseteq \qvar(\A)$.
\end{proof}

\begin{corollary} \label{cor3}
Let $S$ be a consistent system of equations over $\A$-algebra
$\B$. Then the coordinate algebra $\Gamma_\B (S)$ is an
$\A$-algebra too.
\end{corollary}

\begin{proof}
All algebras from $\pvar_\A (\B)$ except the trivial algebra $\E$
are $\A$-algebras~\cite[Corollary~3.16]{DMR1}. If $\Gamma _\B
(S)=\E$ and $S$ is consistent then the empty set $\emptyset$ is
not algebraic over $\B$. Hence, by Lemma~\ref{emptyset}, $\B$ has
a trivial $\L_\A$-subalgebra. It is possible if and only if
$\A\cong \E$.
\end{proof}

The asserted connection between classification of algebraic sets
upto isomorphism and classification of them coordinate algebras
will be discussed in Section~\ref{sec:dual}, and now let us prove
two first results in this direction.

\begin{lemma}\label{cor0}
Let $Y$ and $Z$ are algebraic sets in $A^n$ such that $Y \subseteq
Z$. Then there exists an epimorphism $h\colon \Gamma (Z) \to
\Gamma (Y)$. Moreover, if the inclusion $Y \subset Z$ is strict
then the epimorphism $h$ is proper.
\end{lemma}

\begin{proof}
As $Y \subseteq Z$ then $\Rad (Y) \supseteq \Rad (Z)$, i.e.,
$\theta_{\Rad (Y)} \geqslant \theta_{\Rad (Z)}$. Hence, there
exists the natural epimorphism $h\colon \Gamma (Z) \to \Gamma
(Y)$. If $Y \ne Z$ then $\Rad (Y) \ne \Rad (Z)$, so the
epimorphism $h$ is not a monomorphism.
\end{proof}

\begin{lemma}\label{cor7}
Let $Y\subseteq A^n$ and $Z\subseteq A^m$ be an algebraic sets
over $\A$. Suppose there exists an epimorphism $h\colon\Gamma
(Z)\to\Gamma (Y)$. Then there exists an algebraic subset
$Y'\subseteq Z$ with $\Gamma(Y)\cong\Gamma(Y')$. Moreover, if $h$
is proper then the inclusion $Y' \subset Z$ is strict.
\end{lemma}

\begin{proof}
Let us introduce the notation $\Gamma
(Z)=\T_\L(\{x'_1,\ldots,x'_m\})/\theta_{\Rad(Z)}$. Since $h$ is an
epimorphism the coordinate algebra $\Gamma (Y)$ generates by the
set $X'=\{h(x'_i/\theta_{\Rad (Z)}), i=\overline{1,m}\,\}$. Thus,
there exists $S'\subseteq \At_\L(X')$ such that $\Gamma
(Y)\cong\langle X' \mid S' \rangle$. It is obvious that $[S']
\supseteq \Rad (Z)$. By Corollary~\ref{cor4}, $[S']$ is a radical
ideal over $\A$. Hence, for the algebraic set $Y'=\V_\A(S')$ we
have $\Rad (Y')=[S']$. It is clear that $\Gamma (Y)\cong \Gamma
(Y')$. Since $[S'] \supseteq \Rad (Z)$ we have inclusion $Y'
\subseteq Z$. Furthermore, if $h$ is proper then the inclusion
$[S'] \supset \Rad (Z)$ is strict, and therefore, the inclusion
$Y' \subset Z$ is also strict.
\end{proof}

\subsection{The Zariski topology and irreducible sets} \label{subsec:topology}

There are three perspectives for investigation in the algebraic
geometry over a given algebra $\A$~--- algebraic, geometric, and
logic. The geometric approach is connected with examination of the
affine space $A^n$ as topological space.

Following~\cite{BMR1}, we define the {\em Zariski topology} on
$A^n$, where algebraic sets over $\A$ form a pre-basis of closed
sets, i.e., closed sets in this topology are obtained from the
algebraic sets by finite unions and arbitrary intersections.

\begin{remark}
Suppose that $\A_1$ and $\A_2$ are algebras with the same universe
set $A$. Let $\mathfrak{T}_i$ be the family of algebraic sets
$Y\subseteq A^n$ over $\A_i$, $i=1,2$. In general,
$\mathfrak{T}_1$ and $\mathfrak{T}_2$ are different families. Then
the affine space $A^n$ possesses two Zariski topologies. For the
sake of good order we assume everywhere below that the language
$\L$ and $\L$-algebra $\A$ with universe $A$ are fixed.
\end{remark}

For a subset $Y \subseteq A^n$ we denote by $\overline{Y}$ its
closure in the Zariski topology on $A^n$ and by $Y^\ac$ the least
algebraic set over $\A$ which contains $Y$. It is clear that
$$
Y^\ac=\V_\A (\Rad (Y))= \bigcap_{Y\subseteq Z}{\{Z,\;\;
Z\;\mbox{is algebraic set over }\A\}}.
$$

In the classic algebraic geometry when $\A$ is a field notions of
$\overline{Y}$ and $Y^\ac$ coincide, because in this case all sets
closed in the Zariski topology are algebraic. Algebraic structures
with such property are called {\it equational domains}. We discuss
equational domains in one of the next articles on the universal
algebraic geometry.

In general case we have only inclusion $\overline{Y} \subseteq
Y^\ac$ for a subset $Y \subseteq A^n$. It is clear that
$\overline{Y} = Y^\ac$ if and only if $\overline{Y}$ is an
algebraic set. Lemma~\ref{irr1} below shows that identity
$\overline{Y} = Y^\ac$ holds for every irreducible set $Y$.

\begin{definition}
A subset  $\emptyset\ne Y\subseteq A^n$ is called {\em
irreducible} if for all closed subsets $Y_1, Y_2 \subseteq A^n$
inclusion $Y \subseteq Y_1\cup Y_2$ implies $Y\subseteq Y_1$ or $Y
\subseteq Y_2$; otherwise, it is called {\em reducible}. The empty
set is not considered to be irreducible.
\end{definition}

For example, any singleton set $\{p\}$, $p\in A^n$, is
irreducible.

A non-empty closed set is irreducible if and only if it is not a
union of two proper closed subsets. In an arbitrary topological
space $(W,\mathfrak{T})$ a subset $Y\subset W$ is irreducible if
and only if its closure $\overline{Y}$ is irreducible. In our case
this fact has the following specification.

\begin{lemma}\label{irr1}
Let $Y\subseteq A^n$ be an irreducible set. Then $\overline{Y} =
Y^\ac$, i.e., $\overline{Y}$ is an irreducible algebraic set over
$\A$. In particularly, every closed irreducible subset $Y\subseteq
A^n$ is algebraic.
\end{lemma}

\begin{proof}
Let $\overline{Y}=\bigcap\limits_{i\in I}{\{Y^i_1 \cup\ldots \cup
Y^i_{m_i}\}}$, where $Y^i_j$ are algebraic sets. For each $i\in I$
we have $Y \subseteq Y^i_1 \cup\ldots \cup Y^i_{m_i}$, hence there
exists $j(i)\in \{1,\ldots,m_i\}$ such that $Y\subseteq
Y^i_{j(i)}$. So, $\overline{Y}=\bigcap\limits_{i\in
I}{Y^i_{j(i)}}$ is an algebraic set, and therefore $\overline{Y} =
Y^\ac$.
\end{proof}

In an arbitrary topological space $(W,\mathfrak{T})$ every
irreducible subset $Y\subset W$ is contained in a maximal
irreducible subset $Z$ (it follows from Zorn Lemma, since the
union of sets from an ascending chain of irreducible sets is an
irreducible set), which is, of cause, is closed (since the closure
of irreducible set is irreducible). The maximal irreducible
subsets $Z\subset W$ are called {\em irreducible components} of
$W$. The irreducible components cover $W$ (since every point $p\in
W$ forms an irreducible set, that is contained in a maximal
irreducible set). In our case this topological fact turns into the
following statement.

\begin{lemma}\label{irr0}
Every non-empty closed in the Zariski topology subset $Y\subseteq
A^n$ is a union of maximal irreducible algebraic over $\A$ subsets
$Y_i\subseteq Y$~--- irreducible components of $Y$.
\end{lemma}

\begin{proof}
Let us cover $Y$ with the induced Zariski topology by its
irreducible components $\{Y_i, i\in I\}$. Each of them is a
maximal irreducible closed subset in $Y$. Then $Y_i$ is closed and
irreducible in $A^n$, and by Lemma~\ref{irr1}, $Y_i$ is an
algebraic set over $\A$.
\end{proof}

\medskip

A more strong result, then Lemma~\ref{irr0}, holds in the
classical algebraic geometry over a field: here every non-empty
closed set is a finite union of irreducible components. In a
general case such result holds if an algebra $\A$ is equationally
Noetherian (we will discuss such algebras in the next
Section~\ref{sec:eq_noeth_algebras}). Anyway, Lemma~\ref{irr0}
shows the importance of studding of irreducible algebraic sets
and, correspondingly, their coordinate algebras.

\begin{lemma}\label{algirr}
Let $Y\subseteq A^n$ be a non-empty algebraic set over $\A$. Then
the following conditions are equivalent:
\begin{itemize}
\item $Y$ is irreducible;
\item $Y$ is not a finite union of proper algebraic
subsets.
\end{itemize}
\end{lemma}

\begin{proof}
It follows from definition that if $Y$ is a finite union of proper
algebraic subsets then it is reducible. Conversely, let us assume
that an algebraic set $Y$ is reducible, i.e., $Y$ is a union of
two closed proper subsets: $Y=Y_1\cup Y_2$. We can write
$Y_1=\bigcap_{i\in I} Z_i$ and $Y_2=\bigcap_{j\in J} W_j$, where
$Z_i, W_j$ are finite unions of algebraic sets. Thus,
$Y=\bigcap_{i\in I, j\in J} Z_i\cup W_j$. Since $Y\ne Y_1$ and
$Y\ne Y_2$, there exist $i \in I$ and $j\in J$ such $Y\nsubseteq
Z_i$ and $Y\nsubseteq W_j$. Therefore, $Y=(Y \cap Z_i)\cup(Y\cap
W_j)$ is a decomposition of algebraic set $Y$ into a finite union
of proper algebraic subsets.
\end{proof}

\begin{corollary}\label{cor14}
The algebraic set $\V_\A(\At_\L(X))$ is either irreducible or the
empty set. Correspondingly, the trivial algebra $\E$ is either the
coordinate algebra of an irreducible algebraic set or $\E=\Gamma
(\emptyset)$.
\end{corollary}

\begin{proposition}\label{irrcoo}
Let $\A$ be an algebra in $\L$. Then for a finitely generated
algebra $\C$ of $\L$ the following conditions are equivalent:
\begin{itemize}
\item[1)] $\C$ is discriminated by $\A$;
\item[2)] $\C$ is the coordinate algebra of an irreducible
algebraic set over $\A$ defined by a system of equations in $\L$.
\end{itemize}
\end{proposition}

\begin{proof}
First of all we consider the case when $\C$ is the trivial algebra
$\E$. By definition $\E$ is discriminated by $\A$ if and only if
$\A$ has a trivial subalgebra. At once, by Lemma~\ref{emptyset},
the trivial algebra $\E$ is the irreducible coordinate algebra of
an irreducible algebraic set over $\A$ if and only if $\A$ has a
trivial subalgebra.

Assume that $\C\ne \E$. Let us prove at first implication $2)
\Longrightarrow 1)$. Suppose to the contrary that $\C$ is the
coordinate algebra of an irreducible algebraic set $Y=\V (S)$,
$\C\simeq \T_\L(X)/\theta_{\Rad(Y)}$, and $\C$ is not
discriminated by $\A$. Thus, there exist atomic formulas
$(t_i=s_i)\in \At_{\L} (X)$, $(t_i=s_i)\not\in \Rad(Y)$, $i=1,
\ldots,m$, such that for any homomorphism $h\colon \C\to \A$ there
exists an index $i\in \{1,\ldots,m\}$ with $h(t_i/\theta_{\Rad
(Y)})=h(s_i/\theta_{\Rad (Y)})$. Hence, for any $p\in Y$ there
exists an index $i\in \{1,\ldots,m\}$ with
$t_i^{\A}(p)=s_i^{\A}(p)$. Put $Y_i = \V (S\cup \{t_i=s_i\})$,
$i=1, \ldots,m$. We have $Y=Y_1 \cup \ldots \cup Y_m$, moreover,
the sets $Y_1,\ldots,Y_m$ are proper closed subsets of $Y$. It
contradicts with the irreducibility of $Y$.

Let us prove $1) \Longrightarrow 2)$. Since  $\Dis(\A)\subseteq
\Res (\A)$ then, by Proposition~\ref{coo}, $\C=\Gamma (Y)$ for
some algebraic set $Y$ over $\A$ ($Y \ne \emptyset$, because
$\C\ne \E$). To prove that $Y$ is irreducible it suffices to
reverse the argument above. Indeed, suppose $Y=Y_1 \cup \ldots
\cup Y_m$ for some proper algebraic subsets $Y_i$. From
$Y_i\subsetneq Y$, by Lemma~\ref{rad}, follows that $\Rad
(Y)\subsetneq \Rad (Y_i)$. So, there exists an atomic formula
$(t_i=s_i)\in \Rad (Y_i) \backslash \Rad (Y)$, $i= 1, \ldots,m$.
This implies that there is no any homomorphism $h\colon \C \to \A$
with $h(t_i/\theta_{\Rad (Y)}) \ne h(s_i/\theta_{\Rad (Y)})$ for
all $i=1, \ldots,m$~--- contradiction with  $\C \in \Dis (\A)$.
\end{proof}

\begin{corollary}
The class of all coordinate algebras of irreducible algebraic sets
over an algebra $\A$ coincides with $\Dis (\A)_\omega$.
\end{corollary}

\section{Equationally Noetherian algebras}
\label{sec:eq_noeth_algebras}

Let $\A$ be an algebra in a functional language $\L$ and $\B$ an
$\A$-algebra.

\begin{definition}
Algebra $\A$ is called {\em equationally Noetherian} (with respect
to $\L$-equations) if for any positive integer $n$ and any system
of equations $S\subseteq \At_{\L}(x_1,\ldots,x_n)$ there exists a
finite subsystem $S_0 \subseteq S$ such that $\V _\A (S)=\V _\A
(S_0)$.
\end{definition}

If an $\A$-algebra $\B$ is equationally Noetherian with respect to
$\A$-equations we say also that it is $\A$-{\em equationally
Noetherian}.

For a given algebra $\A$: how we can establish is $\A$
equationally Noetherian or not? The natural way to answer this
question is to examine all systems of equations $S$ in order to
check whether $S$ is equivalent to some its finite subsystem or
not. N.\,S.\,Romanovskii has called our attention to the question:
should we check inconsistent systems as well as consistent
systems? As usual, in concrete algebraic structures the
examination of inconsistent systems is trivial. However, the
following problem is natural.

\begin{problem}
Find an algebra $\A$ such that every consistent over $\A$ is
equivalent to its finite subsystem and there exists inconsistent
system which is not equivalent to some finite subsystem over $\A$.
\end{problem}

In~\cite{Shevl1} A.\,N.\,Shevlyakov has constructed an example of
commutative idempotent semigroup $\A$ in the language with
countable set of constants which is ``equationally Noetherian with
respect to consistent systems'', but is not equationally
Noetherian.

The following statement gives the alternative ways to examine
whether $\A$ is equationally Noetherian.

\begin{statement}
For an $\L$-algebra $\A$ the following conditions are equivalent:
\begin{itemize}
\item[1)] $\A$ is equationally Noetherian;
\item[2)] for any finite set $X$ and any system $S\subseteq \At_\L(X)$ there
exists finite system $S_0 \subseteq [S]$ such that
$\V_\A(S)=\V_\A(S_0)$;
\item[3)] for any positive integer $n$ the Zariski topology on  $A^n$ is
Noetherian, i.e., it satisfies the descending chain condition on
closed subsets;
\item[4)] for any positive integer $n$ every chain
$$
Y_1 \; \supset \; Y_2  \; \supset \; Y_3 \; \supset \;\ldots
$$
of distinct algebraic sets in $A^n$ is finite;
\item[5)] every chain
$$
\Gamma (Y_1) \; \to \; \Gamma (Y_2) \; \to \; \Gamma (Y_3) \; \to
\; \ldots
$$
of proper epimorphisms of coordinate algebras of algebraic sets
$Y_1, Y_2, Y_3,\ldots$ over $\A$ is finite;
\item[6)] for any finite set $X$ the set of atomic formulas $\At_\L(X)$ satisfies the
ascending chain condition on radical ideals over $\A$.
\end{itemize}
\end{statement}

\begin{proof}
Implication $1) \Longrightarrow 2)$ is trivial. To show $2)
\Longrightarrow 1)$ note that for every atomic formula $c=(t=s)\in
[S]$ there exists a finite subsystem $S_c\subseteq S$ such that
$S_c \vdash (t=s)$. Therefore, if $\V_\A(S)=\V_\A(S_0)$ for a
finite system $S_0 \subseteq [S]$ then one has
$\V_\A(S)=\V_\A(\bigcup_{c\in S_0} {S_c})$.

Equivalencies $1) \Longleftrightarrow 3)$ and $3)
\Longleftrightarrow 4)$ have been proven in~\cite[Lemma~4.11 and
Remark~4.8]{DMR1}.

Implication $5) \Longrightarrow 4)$ follows from Lemma~\ref{cor0},
and the converse implication $4) \Longrightarrow 5)$~--- from
Lemma~\ref{cor7}. Equivalence $4) \Longleftrightarrow 6)$ follows
from Lemma~\ref{rad}.
\end{proof}

As it has been announced, for equationally Noetherian algebras
Lemma~\ref{irr0} attains the form of well-know theorem from
classical algebraic geometry.

\begin{theorem}[\cite{DMR1}]\label{irr}
Let $\A$ be an equationally Noetherian algebra. Then any non-empty
closed in the Zariski topology subset $Y\subseteq A^n$ (in
particularly, any non-empty algebraic set) is a finite union of
irreducible algebraic sets (irreducible components): $Y=Y_1 \cup
\ldots \cup Y_m$. Moreover, if  $Y_i \not \subseteq Y_j$ for $i\ne
j$ then this decomposition is unique up to a permutation of
components.
\end{theorem}

Denote by $\EN$ the class of all equationally Noetherian
$\L$-algebras. In Section~\ref{sec:preliminaries} it is presented
the list of operators. What operators from that list image $\EN$
to $\EN$?

\begin{statement}\label{noe}
Let $\A$ be an equationally Noetherian $\L$-algebra. Then the
following $\L$-algebras are equationally Noetherian too:
\begin{itemize}
\item[1)] every subalgebra of $\A$;
\item[2)] every filterpower, direct power, ultrapower of $\A$;
\item[3)] the coordinate algebra $\Gamma (Y)$ of an algebraic set $Y$ over
$\A$;
\item[4)] every algebra which is separated or discriminated by $\A$;
\item[5)] every algebra from $\qvar (\A)$, $\ucl (\A)$;
\item[6)] every limit algebra over $\A$;
\item[7)] every finitely generated algebra defined by a complete atomic type in the theory $\Th _{\rm qi} (\A)$ or $\Th _{\forall} (\A)$.
\end{itemize}
\end{statement}

\begin{proof}
Item~1) is obvious. Item~2) follows from Lemma~\ref{up}. Item~3)
follows from items~1),~2) and Proposition~\ref{coo}. Item~4) is
true because of $\Dis(\A)\subseteq\Res(\A)=\S\P(\A)$~\cite{DMR1}.
Item~5) follows from
$\ucl(\A)\subseteq\qvar(\A)=\S\Pf(\A)\e$~\cite{Malcev}. Every
limit algebra over $\A$ embeds into an ultrapower of
$\A$~\cite[Corollary~5.7]{DMR1}, that proves item~6). Every
finitely generated algebra defined by a complete atomic type in
the theory $\Th _{\rm qi} (\A)$ (or $\Th _{\forall} (\A)$) belongs
to $\qvar (\A)$ (or $\ucl(\A)$)~\cite[Lemma~4.7]{DMR1}. Thus, we
have item~7).
\end{proof}

So, the class $\EN$ is closed under ultrapowers.

\begin{problem}
Is the class $\EN$ closed under ultraproducts?
\end{problem}

As $\EN$ is closed under taking subalgebras the problem above is
equivalent to the following problem~\cite{Malcev}.

\begin{problem}
Is the class $\EN$ axiomatizable?
\end{problem}

The negative solution of this problem has been presented
in~\cite{Shevl1} for the class of equationally Noetherian
commutative idempotent semigroups in the language with countable
set of constants.

\begin{example}[positive examples]
Every algebra $\A$ in the list below is $\A$-equationally
Noetherian:
\begin{itemize}
\item a Noetherian commutative ring;
\item a linear group over a Noetherian
ring (in particular, a free group, a polycyclic group, a finitely
generated metabelian group)~\cite{BMR1, Bryant, Guba};
\item a torsion-free hyperbolic group~\cite{Sela3};
\item a free solvable group~\cite{GRom};
\item a finitely generated metabelian (or nilpotent) Lie
algebra~\cite{Daniyarova1}.
\end{itemize}
About equationally Noetherian property for the universal
enveloping algebras of wreathe products of abelian Lie algebras
see~\cite{RomSh}.
\end{example}

\begin{example}[negative examples]
The following algebras are not equationally Noetherian:
\begin{itemize}
\item infinitely generated nilpotent groups~\cite{MR2};
\item the wreath product $A\wr B$ of a non-abelian group $A$ and an infinite group $B$~\cite{BMRom};
\item the min-max structures $\mathcal{M}_\mathbb{R}=\langle \mathbb{R}; {\max},
{\min}, {\cdot}, \allowbreak {+}, {-}, 0, 1\rangle$ and \\
$\mathcal{M}_\mathbb{N}=\langle \mathbb{N}; {\max}, {\min},
\allowbreak {+}, 0, 1\rangle$~\cite{DvorKot}.
\end{itemize}
\end{example}

\begin{lemma}\label{trucl}
Let $\A$ be an equationally Noetherian $\L$-algebra. The universal
closure $\ucl (\A)$ contains the trivial algebra $\E$ if and only
if $\A$ has a trivial subalgebra.
\end{lemma}

\begin{proof}
It is clear that $\E\leq \A$ implies $\E \in \ucl (\A)$. Suppose
$\A$ has not a trivial subalgebra. By Lemma~\ref{emptyset}, $\V
(\At_\L(x))=\emptyset$. Hence, there exists a finite system $S_0
\subset \At_\L (x)$ such that $\V (S_0)=\emptyset$, i.e., the
following universal formula
\begin{equation}\label{eq:e}
\forall \; x\quad \left(\bigvee_{(t=s)\in S_0} {t(x)\ne s(x)}
\right)
\end{equation}
holds in $\A$. However,~\eqref{eq:e} is false in $\E$, so $\E
\not\in \ucl (\A)$.
\end{proof}

\section{The theorem on duality of the category of algebraic sets and the category of coordinate algebras} \label{sec:dual}

Let $\L$ be a functional language and $\A$ an algebra in $\L$.

In Subsection~\ref{subsec:categories} we introduce two categories:
the category $\AS (\A)$ of algebraic sets over $\A$ and the
category $\CA (\A)$ of coordinate algebras of algebraic sets over
$\A$. In Subsection~\ref{subsec:theorem_dual_equiv} we prove that
these categories are dually equivalent (Theorem~\ref{dual}). In
Subsection~\ref{subsec:classification} we discuss how
Theorem~\ref{dual} is useful when classifying algebraic sets.

\subsection{The category of algebraic sets and the category of coordinate algebras}\label{subsec:categories}

Objects of $\CA (\A)$ are all coordinate algebras of algebraic
sets over $\A$. For two coordinate algebras $\C_1$ and $\C_2$ in
$\CA (\A)$ the set of morphisms $\Mor(\C_1,\C_2)$ coincides with
the set $\Hom (\C_1,\C_2)$ of all $\L$-homomorphisms from $\C_1$
into $\C_2$. Note that the trivial algebra $\E$ is the terminal
object in $\CA (\A)$. It means that for every object $\C$ in $\CA
(\A)$ there is exactly one morphism from $\C$ to $\E$.

Objects of $\AS (\A)$ are all algebraic sets over $\A$. To define
morphisms in $\AS (\A)$ we need the notion of a term map.

\begin{definition}
A map $\varphi\colon A^n \to A^m$ is called a {\em term map} if
there exist terms $t_1,\ldots,t_m \in \Tr_\L (x_1,\ldots,x_n)$
such that
\begin{equation}\label{varphi}
\varphi
(a_1,\ldots,a_n)=(\,t^\A_1(a_1,\ldots,a_n)\,,\,\ldots\,,\,t^\A_m(a_1,\ldots,a_n)\,)
\end{equation}
for all $(a_1,\ldots,a_n)\in A^n$. For two non-empty algebraic
sets $Y\subseteq A^n$ and $Z\subseteq A^m$ a map $\varphi\colon Y
\to Z$ is called a {\em term map} if it is a restriction on $Y$ of
some term map $\varphi\colon A^n \to A^m$ such that $\varphi (Y)
\subseteq Z$.
\end{definition}

\begin{remark}
Note that a term map $\psi\colon A^n\to A^m$ defined by terms
$s_1,\ldots,s_m \in \Tr_\L (x_1,\ldots,x_n)$ may induce the same
term map $\psi\colon Y\to Z$ as $\varphi\colon Y \to Z$ above. It
happens if and only if $(t_i=s_i)\in \Rad (Y)$ for all
$i=1,\ldots,m$. Thereby, any term map $\varphi\colon Y \to Z$ is
uniquely defined by an ordering set of term functions
$t^Y_1,\ldots,t^Y_m \in \Tf (Y)$ such that
$(t^Y_1(p),\ldots,t^Y_m(p))\in Z$ for all $p\in Y$.
\end{remark}

We put the family $\Mor (Y, Z)$ of morphism from object $Y$ to
object $Z$ in $\AS (\A)$ coincides with the set of all term maps
$\varphi\colon Y \to Z$. Furthermore, $\id _Y$ is the identical
map on $Y$. If the empty set $\emptyset$ is algebraic over $\A$ we
place it into $\AS (\A)$ as the initial object. It means that for
every object $Y$ in $\AS (\A)$ there is exactly one arrow
(morphism) from $\emptyset$ to $Y$.

As usual, one can define the notion of an isomorphism in the
categories $\CA (\A)$ and $\AS (\A)$. Thus, algebraic sets
$Y\subseteq A^n$ and $Z\subseteq A^m$ are isomorphic if and only
if there exist term maps $\varphi\colon Y \to Z$ and $\psi\colon Z
\to Y$ such that $\psi \circ \varphi =\id_Y$ and $\varphi \circ
\psi=\id_Z$.

\begin{example}
Let $\L=\{\cdot, ^{-1}, e\}$ be the language of groups, $G$ a
group and $Y\subseteq G^n$ an algebraic set over $G$ for a system
of equations in the extended language $\L_G$ (with coefficients in
$G$). Then for every element $h\in G$ the shift
$$
Yh=\{(g_1 h,\ldots, g_n h) \mid (g_1,\ldots, g_n) \in Y\}
$$
of $Y$ is an algebraic set over $G$ which is $\L_G$-isomorphic to
 $Y$. Indeed, if $Y=\V (S(x_1,\ldots,x_n))$ then
$Yh=\V(S(x_1h^{-1},\ldots,x_nh^{-1}))$. Isomorphism between $Y$
and $Yh$ is established by term maps $\varphi, \psi\colon G^n\to
G^n$:
$$
\varphi (g_1,\ldots, g_n)=(g_1 h,\ldots, g_n h), \quad \psi
(g_1,\ldots, g_n)=(g_1 h^{-1},\ldots, g_n h^{-1}).
$$
It is evident that $\varphi \circ \psi =\id _{G^n}$ and $\psi
\circ \varphi = \id _{G^n}$.
\end{example}

\begin{lemma}\label{termmap}
Let $\varphi\colon A^n \to A^m$ be a term map. Then the following
holds:
\begin{itemize}
\item[1)] If $Z$ is an algebraic set in $A^m$ then
$\varphi^{-1}(Z)$ is an algebraic set in $A^n$.
\item[2)] The map $\varphi$ is continuous in the Zariski topology.
\item[3)] If $Y$ is an irreducible subset in $A^n$ then $\varphi (Y)$ is an irreducible subset in
$A^m$.
\item[4)] Isomorphic algebraic sets are irreducible and reducible
simultaneously.
\end{itemize}
\end{lemma}

\begin{proof}
Let $t_1,\ldots,t_m$ be the terms from~\eqref{varphi}. Suppose
that $Z=\V (S')$, where $S'\subseteq \At_\L(x'_1,\ldots,x'_m)$.
Taking $S=S'(t_1(x_1,\ldots,x_n),\ldots,t_m(x_1,\ldots,x_n))$ we
have $\varphi^{-1}(Z)=\V(S)$. It proves item~1). Item~2) follows
from item~1), since algebraic sets form a closed pre-base of the
Zariski topology. Every continuous map between topological spaces
images irreducible sets into irreducible ones, so we have item~3).
Item~4) easy follows from item~3).
\end{proof}

The following result takes place in Diophantine algebraic
geometry. Its proof is similar to the proof of the corresponding
result in the classical algebraic geometry over a
field~\cite{Shafarevich}.

\begin{lemma}
Let $\A$ be an $\L$-algebra and $Y\subseteq A^n$, $Z\subseteq A^m$
algebraic sets over $\A$ defined by systems of equations with
coefficients in $\A$. The algebraic set $Y\times Z$ is irreducible
if and only if $Y$ and $Z$ are irreducible (irreducibility is
considered with respect to the Zariski topology for $\A$ as
$\L_\A$-algebra).
\end{lemma}

\begin{proof}
Suppose that $Y$ is a reducible algebraic set, i.e., $Y$ is a
finite union of proper algebraic subsets: $Y=Y_1\cup \ldots \cup
Y_d$. It follows that $Y \times Z = (Y_1 \times Z) \cup \ldots
\cup (Y_d \times Z)$ is a decomposition of $Y\times Z$ into the
finite union of proper algebraic subsets, so $Y\times Z$ is
reducible.

Assume now that $Z$ is an irreducible algebraic set and $Y \times
Z = W_1 \cup \ldots \cup W_d$ is a decomposition of $Y\times Z$
into a finite union of proper algebraic subsets. We show that $Y$
is reducible in this case.

Every point $p\in A^n$ forms an algebraic set $\{p\}$ over $\A$.
Moreover, algebraic sets $Z$ and $\{p\}\times Z$ are isomorphic.
In particularly, $\{p\}\times Z$ is an irreducible algebraic set.
If $p\in Y$ then $\{p\}\times Z \subseteq W_1 \cup \ldots \cup
W_d$. It implies that $\{p\}\times Z \subseteq W_i$ for some $i\in
\{1,\ldots,d\}$. Denote by $Y_i$ the set $\{p\in Y | \{p\}\times Z
\subseteq W_i\}$, $i=\overline{1,d}$. One has $Y=Y_1 \cup \ldots
\cup Y_d$ and $Y \ne Y_i$ for all $i=\overline{1,d}$.

Let us check that $Y_i$ is an algebraic set for each
$i=\overline{1,d}$. For a point $p'\in Z$ denote by $Y_{i,p'}$ the
set $\{p\in Y | \{p\}\times\{p'\}\subseteq W_i\}$. As the set
$(Y\times \{p'\})\cap W_i$ is algebraic and $Y_{i,p'}\times
\{p'\}=(Y\times \{p'\})\cap W_i$, therefore $Y_{i,p'}$ is an
algebraic set. Finally, note that $Y_i=\bigcap_{p'\in
Z}{Y_{i,p'}}$. Hence, $Y_i$ is an algebraic set.
\end{proof}

\subsection{The theorem on dual equivalence} \label{subsec:theorem_dual_equiv}

This subsection is required the basic notions and ideas of
category theory. We refer to~\cite{Barr} in every way.

\begin{theorem}\label{dual}
The category $\AS (\A)$ of algebraic sets over an algebra $\A$ and
the category $\CA (\A)$ of coordinate algebras of algebraic sets
over $\A$ are dually equivalent.
\end{theorem}

\begin{proof}
To prove the theorem we need to construct a contravariant functor
$\F\colon \AS (\A) \to \CA (\A)$, i.e., a map such that
\begin{itemize}
\item[F1)] if $\varphi\colon Y \to Z$ is a morphism of $\AS (\A)$ then $\F(\varphi) \colon \F(Z) \to
\F(Y)$ is a morphism of $\CA (\A)$;
\item[F2)] $\F(\id _Y)=\id_{\F(Y)}$ for every object $Y$ of $\AS (\A)$;
\item[F3)] if $\psi \colon Z \to W$ is a morphism of $\AS (\A)$ then $\F(\psi \circ
\varphi)=\F(\varphi) \circ \F(\psi)$.
\end{itemize}
After that we need to show that $\F$ is a dual equivalence. There
are several equivalent definitions of dual equivalence
in~\cite{Barr}. We take one the most convenient for our needs. A
functor $\F$ is a dual equivalence if
\begin{itemize}
\item[E1)] $\F$ is fully faithful, i.e., for every objects $Y$ and $Z$ of $\AS (\A)$
and every morphism $h\in\Hom (\F(Z), \F(Y))$ there is one and only
one morphism $\varphi \in \Hom (Y,Z)$ such that $h=\F(\varphi)$;
\item[E2)] $\F$ is representative, i.e., for any object $\C$ of $\CA (\A)$ there is an object $Y$ of
$\AS (\A)$ for which $\F(Y)$ is isomorphic to $\C$.
\end{itemize}

To define the functor $\F$ we put $\F(Y)=\Gamma (Y)$ for an
algebraic set $Y$ of $\AS (\A)$. Also we have to define $\F$ on
morphisms. Let $Y$ and $Z$ are objects of $\AS (\A)$. If
$Y=\emptyset$ then $\Gamma (Y)=\E$. Furthermore, $\Mor (Y,Z)$ has
a unique arrow $\varphi$ and $\Hom (\Gamma (Z), \E)$ has a unique
morphism $h$. So we put $\F(\varphi)=h$. If $Y\ne \emptyset$ and
$Z=\emptyset$ then $\Mor (Y,Z)=\emptyset$.

Suppose now that $Y\subseteq A^n$ and $Z\subseteq A^m$ non-empty
algebraic sets in $\AS (\A)$ and $\varphi \in \Mor (Y,Z)$ a
morphism, defined by term functions $t^Y_1,\ldots,t^Y_m\in \Tf
(Y)$. For defining the morphism $\F (\varphi)\in \Hom (\Gamma (Z),
\Gamma (Y))$ it will be convenient to think about coordinate
algebras as term functions algebras, that is possible due to
Lemma~\ref{termfunc}. Algebra $\Gamma (Z) \simeq \Tf (Z)$
generates by coordinate term functions $x^Z_1,\ldots,x^Z_m$ on
$Z$, so it is sufficient to define morphism $h=\F(\varphi)$ on
these generators. Let us put
\begin{equation}\label{eq:h}
h(x^Z_1)=t^Y_1,\;  \ldots, \; h(x^Z_m)=t^Y_m.
\end{equation}
It is necessary to show that the morphism $h$ is well-defined,
i.e., for each atomic formula $t=s$ in $m$ variables
$t(x^Z_1,\ldots,x^Z_m)=s(x^Z_1,\ldots,x^Z_m)$ implies
$t(t^Y_1,\ldots,t^Y_m)=s(t^Y_1,\ldots,t^Y_m)$. Identity
$t(x^Z_1,\ldots,x^Z_m)=s(x^Z_1,\ldots,x^Z_m)$ means that $(t=s)\in
\Rad (Z)$, and since $\varphi (Y)\subseteq Z$, then
\begin{equation}\label{eq:ts}
t(t^Y_1(p),\ldots,t^Y_m(p))=s(t^Y_1(p),\ldots,t^Y_m(p))
\end{equation}
for all $p\in Y$. So we have that required.

It is not hard to see that F1), F2), F3), and E2) hold. Let us
check E1). Suppose $\varphi\colon Y \to Z$, defined by term
functions $t^Y_1,\ldots, t^Y_m \in \Tf(Y)$, and $\psi\colon Y \to
Z$, defined by term functions $s^Y_1,\ldots, s^Y_m\in \Tf(Y)$, are
distinct morphisms of $\AS (\A)$. Hence, there exists $i\in
\{1,\ldots,m\}$ such that $t^Y_i\ne s^Y_i$. So, $h=\F(\varphi)$
and $g=\F (\psi)$ are distinct homomorphisms, because
$t^Y_i=h(x^Z_i)\ne g(x^Z_i)=s^Y_i$. Hence, functor $\F$ is
faithful.

To establish that $\F$ is full consider an arbitrary homomorphism
$h\colon \Tf (Z) \to \Tf (Y)$ defined by~\eqref{eq:h}. Term
functions $t^Y_1,\ldots, t^Y_m$ define term map $\varphi\colon Y
\to A^m$. Since $h$ is well-defined, for every atomic formula
$(t=s)\in \Rad (Z)$ and each point $p\in Y$ we have
identity~\eqref{eq:ts}, therefore, $\varphi (Y) \subseteq Z$.
Hence, $\F(\varphi)=h$, and $\F$ is full.
\end{proof}

\begin{corollary}\label{cor5}
Two algebraic sets $Y$ and $Z$ over algebra $\A$ are isomorphic if
and only if $\Gamma (Y) \cong \Gamma (Z)$.
\end{corollary}

\begin{proof}
Indeed, every fully faithful functor preserves and reflects
isomorphisms~\cite{Barr}.
\end{proof}

\begin{corollary}
Let $Y$ and $Z$ be non-empty algebraic sets over algebra $\A$.
There exists one-to-one correspondence $\F$ between term maps from
$\Mor (Y, Z)$ and $\L$-homomorphisms from $\Hom(\Gamma (Z),\Gamma
(Y))$.
\end{corollary}

\begin{definition}
We say that an $\L$-algebra $\C$ is an {\it irreducible coordinate
algebra over an $\L$-algebra $\A$} if $\C \cong \Gamma (Y)$ for
some irreducible algebraic set $Y$ over $\A$.
\end{definition}

If $\C \cong \Gamma (Y)$ and $\C \cong \Gamma (Z)$ then algebraic
sets $Y$ and $Z$ are isomorphic, by Corollary~\ref{cor5}. By
Lemma~\ref{termmap}, isomorphic algebraic sets are irreducible and
reducible simultaneously. Thus, irreducible coordinate algebras
are well-defined.


\begin{lemma}
Let $Y\subseteq A^n$ and $Z\subseteq A^m$ be algebraic sets over
an algebra $\A$, $\varphi \in \Mor (Y,Z)$ and $h\in \Hom (\Gamma
(Z), \Gamma (Y))$ morphisms such that $\F(\varphi)=h$:
\begin{equation*}
\begin{CD}
Y @>>\varphi >  Z \\
@V{\F}VV   @  VV{\F}V\\
\Gamma(Y) @< h << \Gamma (Z)
\end{CD}
\end{equation*}
Then the following holds:
\begin{itemize}
\item[1)] If $h$ is an epimorphism then $\varphi$ is a monomorphism.
\item[2)] If $\varphi$ is an epimorphism then $h$ is a
monomorphism.
\item[3)] Furthermore, $h$ is a monomorphism if and only if $\varphi (Y)^\ac=Z$ (see Subsection~\ref{subsec:topology}).
\item[4)] Suppose $Y$ is irreducible. Then $h$ is a monomorphism if and only if
$\overline{\varphi (Y)}=Z$.
\end{itemize}
\end{lemma}

\begin{proof}
For reasoning of the first two statements it is possible to refer
to the appropriate results from the category theory, however, we
prefer to give a direct proofs. By Lemma~\ref{termfunc}, we may
think about $\Gamma (Y)$ and $\Gamma (Z)$ as algebras of term
functions $\Tf (Y)$ and $\Tf (Z)$.

At first, suppose that $\varphi$ is not a monomorphism. Then there
exist distinct points $p_1,p_2\in Y$ such that $\varphi
(p_1)=\varphi (p_2)$. Let $p_1=(a^1_1,\ldots,a^1_n)$ and
$p_2=(a^2_1,\ldots,a^2_n)$. We may assume that $a^1_1\ne a^2_1$.
Denote by $x^Y_1$ the term function $x^Y_1\colon Y \to A$ defined
by term $x_1\in \Tr_\L (x_1,\ldots, x_n)$. Then for an arbitrary
term function $t^Z\in \Tf (Z)$ we have $h(t^Z)(p_1)=h(t^Z)(p_2)$,
so $h(t^Z) \ne x^Y_1$. Hence, $x^Y_1 \not\in h(\Tf (Z))$, and $h$
is not an epimorphism. It proves item~1).

Item~2) follows from item~3). Let us prove~3). By definition $h$
is injective if
$$
t^Z=s^Z \quad \Longleftrightarrow \quad h(t^Z)=h(s^Z) \quad
\mbox{for all } t^Z,s^Z\in \Tf(Z).
$$
The identity $\varphi (Y)^\ac=Z$ is equivalent to
$$
t^Z=s^Z \quad \Longleftrightarrow \quad t^{\,\varphi
(Y)^\ac}=s^{\,\varphi (Y)^\ac} \quad \mbox{for all }
t,s\in\Tr_\L(x_1,\ldots,x_m).
$$
Furthermore, for arbitrary $t,s\in\Tr_\L(x_1,\ldots,x_m)$ one has
\begin{gather*}
h(t^Z)=h(s^Z) \quad \Longleftrightarrow \quad
t(t^Y_1(p),\ldots,t^Y_m(p))=s(t^Y_1(p),\ldots,t^Y_m(p)) \quad \mbox{for all } p\in Y \quad \Longleftrightarrow \\
\Longleftrightarrow \quad t(\varphi(p))=s(\varphi(p)) \quad
\mbox{for all } p\in Y \quad \Longleftrightarrow \quad
(t=s)\in \Rad (\varphi (Y)) \quad \Longleftrightarrow \\
\Longleftrightarrow \quad (t=s)\in \Rad (\varphi (Y)^\ac)\quad
\Longleftrightarrow \quad  t^{\,\varphi (Y)^\ac}=s^{\,\varphi
(Y)^\ac},
\end{gather*}
where $t^Y_1,\ldots, t^Y_m$ are term functions that define the
morphism $\varphi$. It implies item~3).

Suppose finally that $Y$ is irreducible. Then $\varphi (Y)$ is
irreducible too, by Lemma~\ref{termmap}, and by Lemma~\ref{irr1},
$\varphi (Y)^\ac=\overline{\varphi(Y)}$.
\end{proof}

\subsection{Classification of algebraic sets and coordinate algebras} \label{subsec:classification}

It is important to remember that one of the major problems of
algebraic geometry over a given algebra $\A$ lies in classifying
algebraic sets over the algebra $\A$ up to isomorphism. According
to Theorem~\ref{dual}, this problem is equivalent to the problem
of classification of coordinate algebras over $\A$.

Suppose we have attained a classification of coordinate algebras
over $\A$. Then algebraic sets over $\A$ may be found as Hom's.
The corresponding idea is explained in Lemma~\ref{hom}. Sometimes
the expression of algebraic sets in terms of Hom's is reasonable,
as in Example~\ref{ex:abel} below, sometimes not. For instance,
there is a simple description of coordinate groups for equations
in one variable over free metabelian group, while corresponding
algebraic sets have no clear representation~\cite{Rem-Rom2}.

Besides description of all algebraic sets over $\A$, it is very
important to find a classification of irreducible algebraic sets
over $\A$ and their coordinate algebras. Lemma~\ref{irr0} shows
that every algebraic set may be decomposed in a union of maximal
irreducible algebraic sets (irreducible components). Moreover, in
the case when $\A$ is an equationally Noetherian algebra such
decomposition is finite and unique by Theorem~\ref{irr}.

Proposition~\ref{coo} is effective for description of coordinate
algebras over $\A$, and Proposition~\ref{irrcoo} is helpful for
description of irreducible coordinate algebras over $\A$. In the
case when $\A$ is an equationally Noetherian algebra it is
possible to take more informative results~--- Unification
Theorems~A and~C (see
Section~\ref{subsec:unification_theorems_noether}).

The following lemma shows a way for description of all coordinate
algebras when having a classification of irreducible coordinate
algebras.

\begin{lemma}\label{cooalg}
A finitely generated $\L$-algebra $\C$ is the coordinate algebra
of an algebraic set over $\L$-algebra $\A$ if and only if it is a
subdirect product of coordinate algebras of irreducible algebraic
sets over $\A$.
\end{lemma}

\begin{proof}
Suppose at first that $Y$ is an algebraic set over $\A$. By
Lemma~\ref{irr0}, there exist irreducible algebraic sets $Y_i$,
$i\in I$, over $\A$ such that $Y=\bigcup_{i\in I}{Y_i}$. Hence, by
Lemma~\ref{rad}, we have $\Rad (Y) = \bigcap_{i \in I}{\Rad
(Y_i)}$. It implies that there exists subdirect embedding $\Gamma
(Y) \to \prod_{i\in I}{\Gamma (Y_i)}$~\cite[Lemma~3.1]{DMR1}. The
converse statement is true by Corollary~\ref{cor8}.
\end{proof}


\begin{corollary}\label{cor11}
Let $\A$ be an equationally Noetherian $\L$-algebra. A finitely
generated $\L$-algebra $\C$ is the coordinate algebra of an
algebraic set over $\A$ if and only if it is a subdirect product
of a finitely many coordinate algebras of irreducible algebraic
sets over $\A$.
\end{corollary}

\begin{proof}
Theorem~\ref{irr} implies that is required.
\end{proof}

The following example is taken from~\cite{MR2}. A classification
of coordinate groups over an abelian group $A$ has been found in
that paper. Furthermore, this classification allows to describe
algebraic sets over $A$. Also in~\cite{MR2} there have been
classified coordinate groups of irreducible algebraic sets over
$A$.

\begin{example}\label{ex:abel}
Let $A$ be a fixed abelian group and $\L_A$ the language of
abelian groups with constants from $A$, i.e., $\L_A=\{+, -, 0, \:
c_a, a \in A\}$. We consider $A$ as $\L_A$-structure.

Recall that the period of an abelian group $A$ is the minimal
positive integer $m$, if it exists, such that $mA=0$; and $\infty$
otherwise. Let $T(A)$ be the torsion part of $A$ and $T(A)\simeq
\oplus_p T_p(A)$ be the primary decomposition of $T(A)$. Here and
below in this Example $p$ is a prime number. Denote by $e(A)$ the
period of $A$, and by $e_p (A)$ the period of $T_p(A)$.

\noindent {\bf Theorem}~\cite{MR2}{\bf.} {\it Let $C$ be a
finitely generated $A$-group. Then $C$ is the coordinate group of
an algebraic set over $A$ if and only if the following conditions
holds:
\begin{enumerate}
\item $C\simeq A \oplus B$, where $B$ is a finitely generated
abelian group;
\item $e(A)=e(C)$ and $e_p(A)=e_p(C)$ for every prime number $p$.
\end{enumerate}
}

Now it is easy to describe an algebraic set $Y$ corresponding to
the coordinate group $C= A \oplus B$. Fix a primary cyclic
decomposition of the group $B$:
$$
B \simeq \langle a_1 \rangle \oplus \ldots \oplus \langle a_r
\rangle \oplus \langle b_1 \rangle \oplus \ldots \oplus \langle
b_t \rangle,
$$
here $a_i$-s are generators of infinite cyclic groups and $b_j$-s
are generators of finite cyclic groups of orders $p_j^{m_j}$. For
positive integer $n$ denote by $A[n]$ the set $\{a\in A \mid
na=0\}$. By Lemma~\ref{hom}, points form algebraic set $Y$ lie in
one-to-one correspondence with $A$-homomorphisms from
$\Hom_A(A\oplus B, A)$, therefore,
$$
Y \; =\; \underbrace{A\; \oplus\; \ldots\; \oplus\; A}_{r}\: \;
\oplus\; \: A[\,p_1^{m_1}]\; \oplus\; \ldots\; \oplus\;
A[\,p_t^{m_t}].
$$

For a positive integer $k$ and a prime number $p$ we denote by
$\alpha_{p^k}(A)$ the dimension, if it exists, of the factor-group
$A[p^k]/A[p^{k-1}]$ as a vector-space over finite field with $p$
elements; and $\infty$ otherwise.

\noindent {\bf Theorem}~\cite{MR2}{\bf.} {\it Let $C$ be a
finitely generated $A$-group. Then $C$ is the coordinate group of
an irreducible algebraic set over $A$ if and only if the following
conditions holds:
\begin{enumerate}
\item $C\simeq A \oplus B$, where $B$ is a finitely generated
abelian group;
\item $e(A)=e(C)$ and $e_p(A)=e_p(C)$ for every prime number $p$;
\item $\alpha_{p^k}(A)=\alpha_{p^k}(C)$ for each prime number $p$
and positive integer $k$.
\end{enumerate}
}
\end{example}

\section{Unification Theorems for equationally Noe\-the\-rian algebras} \label{subsec:unification_theorems_noether}

The following Unification Theorems help to describe coordinate
algebras of algebraic sets. We first formulate the theorems and
then prove them.

Fix a functional language $\L$.

\medskip

\noindent {\bf Theorem~A.} {\it Let $\A$ be an equationally
Noetherian algebra in $\L$. Then for a finitely generated algebra
$\C$ of $\L$ the following conditions are equivalent:
\begin{enumerate}
\item [1)] $\Th_{\forall} (\A) \subseteq \Th_{\forall} (\C)$, i.e., $\C \in \ucl(\A)$;
\item [2)] $\Th_{\exists} (\A) \supseteq \Th_{\exists} (\C)$;
\item [3)] $\C$ embeds into an ultrapower of $\A$;
\item [4)] $\C$ is discriminated by $\A$;
\item [5)] $\C$ is a limit algebra over $\A$;
\item [6)] $\C$ is an algebra defined by a complete atomic type
in the theory $\Th _{\forall} (\A)$ in $\L$;
\item [7)] $\C$ is the coordinate algebra of an
irreducible algebraic set over $\A$ defined by a system of
equations in the language $\L$.
\end{enumerate}
}

The following Theorem~B is a particular case of Theorem~A for
$\L=\L_\A$. We present it especially for needs of Diophantine
algebraic geometry and algebraic geometry with coefficients in
algebra $\A$.

\medskip

\noindent {\bf Theorem~B} (With coefficients in $\A$){\bf.} {\it
Let $\A$ be an algebra in a functional language $\L$ and $\B$ an
$\A$-equationally Noetherian $\A$-algebra. Then for a finitely
generated $\A$-algebra $\C$ the following conditions are
equivalent:
\begin{enumerate}
\item [1)] $\Th_{\forall,\A} (\B)\subseteq \Th _{\forall,\A} (\C)$, i.e., $\C \in \ucl_\A(\B)$;
\item [2)] $\Th_{\exists,\A} (\B)\supseteq \Th _{\exists,\A} (\C)$;
\item [3)] $\C$ $\A$-embeds into an  ultrapower of $\B$;
\item [4)] $\C$ is $\A$-discriminated by $\B$;
\item [5)] $\C$ is an $\A$-limit  algebra   over  $\B$;
\item [6)] $\C$ is an algebra defined by a complete atomic type in the theory $\Th_{\forall,\A} (\B)$
in the language $\L_{\A}$;
\item [7)] $\C$ is the  coordinate algebra of an irreducible
algebraic set over $\B$ defined by  a system of equations with
coefficients in $\A$.
\end{enumerate}
}

\begin{remark}
In Diophantine algebraic geometry, when $\A=\B$, the first two
items in Theorem~B can be formulated in a more precise form: $\C
\equiv _{\forall,\A} \A$, and  $\C \equiv _{\exists,\A} \A$,
correspondingly. The notation $\C \equiv _{\forall,\A} \A$ implies
that any universal sentence in the language $\L_\A$ holds in $\C$
if and only if it holds in $\A$.
\end{remark}

Theorem~A gives the description of irreducible coordinate
algebras. The following Theorem~C offers the description of all
coordinate algebras.

\medskip

\noindent {\bf Theorem~C.} {\it Let $\A$ be an equationally
Noetherian algebra in $\L$. Then for a finitely generated algebra
$\C$ of $\L$ the following conditions are equivalent:
\begin{enumerate}
\item [1)] $\C \in \qvar(\A)$, i.e., $\Th_{\rm qi} (\A) \subseteq \Th_{\rm qi} (\C)$;
\item [2)] $\C \in \pvar(\A)$;
\item [3)] $\C$ embeds into a direct power of $\A$;
\item [4)] $\C$ is separated by $\A$;
\item [5)] $\C$ is a subdirect product of a finitely many limit algebras over $\A$;
\item [6)] $\C$ is an algebra defined by a complete atomic type in the theory $\Th _{\rm qi} (\A)$ in $\L$;
\item [7)] $\C$ is the coordinate algebra of an algebraic set over $\A$ defined by a system of equations in the language $\L$.
\end{enumerate}
}

The following Theorem~D is a particular case of Theorem~C, as well
as Theorem~B for Theorem~A.

\medskip

\noindent {\bf Theorem~D} (With coefficients in $\A$){\bf.} {\it
Let $\A$ be an algebra in a functional language $\L$ and $\B$ an
$\A$-equationally Noetherian $\A$-algebra. Then for a finitely
generated $\A$-algebra $\C$ the following conditions are
equivalent:
\begin{enumerate}
\item [1)] $\C \in \qvar_\A(\B)$, i.e., $\Th_{\rm qi,\A} (\B) \subseteq \Th _{\rm qi,\A} (\C)$;
\item [2)] $\C \in \pvar_\A(\B)$;
\item [3)] $\C$ $\A$-embeds into a direct power of $\B$;
\item [4)] $\C$ is $\A$-separated by $\B$;
\item [5)] $\C$ is a subdirect product of a finitely many $\A$-limit algebras over $\B$;
\item [6)] $\C$ is an algebra defined by a complete atomic type
in the theory $\Th_{\rm qi,\A} (\B)$ in the language $\L_{\A}$;
\item [7)] $\C$ is the coordinate algebra of an algebraic
set over $\B$ defined by  a system of equations with coefficients
in $\A$.
\end{enumerate}
}

\begin{remark}
In Diophantine algebraic geometry, when $\A=\B$, the first two
items in Theorem~D can be formulated in the form $\qvar_\A(\A)=
\qvar_\A(\C)$, and $\pvar_\A(\A)= \pvar_\A(\C)$, correspondingly.
\end{remark}

\begin{corollary}\label{cor12}
If an algebra $\A$ is equationally Noetherian then the following
identities hold:
\begin{gather*}
\ucl (\A)_\omega = \Dis (\A)_\omega,
\\ \qvar(\A)_\omega=\pvar(\A)_\omega, \quad \qvar (\A)_\omega= \Pw (\ucl
(\A)_\omega).
\end{gather*}
\end{corollary}

\begin{proof}
The first identity follows from equivalence $1)
\Longleftrightarrow 4)$ in Theorem~A, the second identity~--- from
equivalence $1) \Longleftrightarrow 2)$ in Theorem~C. The third
identity follows from equivalence $1) \Longleftrightarrow 5)$ in
Theorem~C and equivalence $1) \Longleftrightarrow 5)$ in
Theorem~A.
\end{proof}

\begin{proof}[Proof of Theorem~A]
Theorem~A has been proven in~\cite{DMR1}. Now we just give more
precise review for the case when $\C$ is the trivial algebra $\E$.
The special case $\C=\E$ has been omitted in the proof
in~\cite{DMR1}. By Proposition~\ref{irrcoo}, the trivial algebra
$\E$ is the coordinate algebra of an irreducible algebraic set
over $\A$ if and only if $\E$ is discriminated by $\A$. By
definition, $\E$ is discriminated by $\A$ if and only if $\A$ has
a trivial subalgebra. By Lemma~\ref{trucl}, $\A$ has a trivial
subalgebra if and only if $\E \in \ucl (\A)$. Thus, items 1), 4),
7) are equivalent. Finally, by Proposition~\ref{logirr}, items 1),
2), 3), 5), 6) are equivalent anyway.
\end{proof}

\begin{proof}[Proof of Theorem~C]
Equivalence $2) \Longleftrightarrow 3)\Longleftrightarrow
4)\Longleftrightarrow 7)$ has been proven in
Proposition~\ref{coo}. Corollary~\ref{cor2} shows that implication
$7) \Longrightarrow 1)$ holds. Equivalence $1) \Longleftrightarrow
6)$ has been proven in~\cite[Lemma~4.7]{DMR1}. Implication $5)
\Longrightarrow 1)$ is easy. Indeed, every limit algebra over $\A$
lies in $\ucl (\A)$~\cite[Corollary~5.7]{DMR1}. Moreover, $\ucl
(\A) \subseteq \qvar (\A)$, and quasivariety $\qvar (\A)$ is
closed under operators $\P$ and $\S$. Thus, if an algebra $\C$ is
a subdirect product of limit algebras over $\A$ then $\C \in \qvar
(\A)$.

By Corollary~\ref{cor11}, the coordinate algebra of an algebraic
set over equationally Noetherian algebra $\A$ is a subdirect
product of a finitely many coordinate algebras of irreducible
algebraic sets over $\A$. According to Theorem~A, coordinate
algebras of irreducible algebraic sets over $\A$ are limit
algebras over $\A$, so we have implication $7) \Longrightarrow
5)$.

Now we prove the last implication $1) \Longrightarrow 4)$. Suppose
that  $\C\not\in \Res (\A)$. It suffices to show that $\C \not\in
\qvar (\A)$. Let $X=\{x_1,\ldots,x_n\}$ be a finite set of
generators of $\C$ and $\langle X \mid S\rangle$ a presentation of
$\C$ in the generators $X$, where $S \subseteq \At_{\L}(X)$. The
latter means that $\C \simeq \T_{\L}(X)/\theta_S$.

Since $\A$ does not separate $\C$, there is an atomic formula
$(t=s)\in \At_{\L} (X)$, $(t=s)\not\in [S]$, such that
$h(t/\theta_S)=h(s/\theta_S)$ for any homomorphism $h\colon \C\to
\A$. This means that  $t^{\A}(p)=s^{\A}(p)$ for any point $p\in
\V_\A(S)$, i.e., $(t=s)\in \Rad_\A (S)$. Since $\A$ is
equationally Noetherian there exists a finite subsystem
$S_0\subseteq S$ such that $\Rad_\A(S_0) = \Rad_\A(S)$. Therefore,
the following qvasi-identity holds in $\A$
\begin{equation}\label{eq:qi}
\forall \;  y_1 \ldots \forall \;  y_n \quad \left( \bigwedge
\limits_{(t_0=s_0) \in S_0} t_0 (\bar{y})=s_0 (\bar{y}) \;
\longrightarrow \; t (\bar{y})=s(\bar{y}) \right) .
\end{equation}
On the other hand the formula
$$
\bigwedge \limits_{(t_0=s_0) \in S_0} t_0 (\bar{y})=s_0 (\bar{y})
\; \longrightarrow \; t (\bar{y})=s(\bar{y})
$$
is false in $\C$ under the interpretation $y_i\mapsto x_i$, $i=1,
\ldots,n$, hence $\C \not\in \qvar (\A)$.

Note that if $\V_\A(S)=\emptyset$, then our reasoning are steel
true. In this case the premise in quasi-identity~\eqref{eq:qi} is
identically false in $\A$.
\end{proof}

Unification Theorems~A and~C are formulated for an equationally
Noetherian algebra $\A$. However, for the reasoning of some
implications in their proofs the equationally Noetherian property
is not required, namely, one has the following remark.

\begin{remark}
The following implications and equivalencies from Theorems~A and~C
hold for an arbitrary algebra $\A$:
$$
\mbox{Theorem~A:} \quad \{4 \Leftrightarrow 7\}\quad
\Longrightarrow \quad \{1 \Leftrightarrow 2\Leftrightarrow
3\Leftrightarrow 5\Leftrightarrow 6\};
$$
$$
\mbox{Theorem~C:} \quad \{5\} \quad \Longrightarrow \quad \{1
\Leftrightarrow 6\} \quad \Longleftarrow \quad \{2 \Leftrightarrow
3 \Leftrightarrow 4 \Leftrightarrow 7\}.
$$
\end{remark}

Theorem~C gives a classification of coordinate algebras over an
equationally Noetherian algebra $\A$ as finitely generated
algebras in the quasivariety $\qvar (\A)$. Therefore, the
following characterizations of quasivariety $\qvar (\K)$ of a
class $\K$ of $\L$-algebras are helpful here:
\begin{gather*}
\qvar (\K)\:=\:\S \Pf (\K)\e\:=\:\S \P \Pu (\K) \:=\: \S \Pu \P
(\K)\:=\:\S \Pu \Pw (\K)\:=\:\\\:=\:\S \Ls \P (\K)\:=\:\Ls \S \P
(\K)\:=\:\Ls \Ps (\K)\:=\:\Ld \S \P (\K).
\end{gather*}

The first one of these identities is due to
Malcev~\cite[\textsection 11, Theorem~4]{Malcev}, and the others
are due to Gorbunov~\cite[Corollary~2.3.4,
Theorem~2.3.6]{Gorbunov}.

\bigskip

We demonstrate an application of Unification Theorems on the
following example.

\begin{example}
Algebraic geometry over the additive monoid of natural numbers has
been studied by P.\,Morar and A.\,Shevlyakov~\cite{MorShev,
Shevl2, Shevl3}. Authors consider $\mathbb{N}$ in several
signatures $\L$. We discuss here the simplest case from these
papers.

Let $\L=\{+, 0\}$ be the basic signature with binary function
$"+"$ and constant $"0"$, and $\N= \langle \mathbb{N}; +,
0\rangle$ the additive monoid of natural numbers in the language
$\L$ with obvious interpretation of symbols from $\L$.

A commutative $\L$-monoid $M$ is called {\it positive} if $x+y=0$
implies $x=y=0$ for all $x,y\in M$. Monoid $M$ is named {\it
monoid with cancelation} if $x+z=y+z$ implies $x=y$ for all
$x,y,z\in M$.

\noindent {\bf Theorem}~\cite{MorShev}{\bf.} {\em For any finitely
generated $\L$-monoid $M$ the following condition are equivalent:
\begin{enumerate}
\item $M$ is the coordinate monoid of an algebraic
set over $\N$;
\item $M$ is separated by $\N$;
\item $M$ is a commutative positive monoid with cancelation;
\item quasi-identities
\begin{gather*}
\forall \; x,y \quad (x+y=y+x),\\
\forall\; x,y \quad (x+y=0\;\to\;x=0),\\
\forall \; x,y,z \quad (x+z=y+z\;\to\;x=z)
\end{gather*}
hold in $M$;
\item $M$ is in the $\ucl (\N)$;
\item $M$ is discriminated by $\N$;
\item $M$ is the coordinate monoid of an irreducible algebraic
set over $\N$.
\end{enumerate}
}

\noindent {\bf Corollary 1.} {\em Every algebraic set over $\N$ is
irreducible.}

\noindent {\bf Corollary 2.} {\em $\ucl(\N)=\qvar(\N)$.}
\end{example}

\section{Open problems and questions}\label{sec:problems}

We do hope that this series of papers on universal algebraic
geometry will be a helpful guide for creating algebraic geometry
over classical algebraic structures. In the nearest future we plan
to publish the papers ``Equationally Noetherian property and
compactness''~\cite{DMR3} and ``Equational domains and
co-domains''~\cite{DMR4} along this series.

As it has been mentioned in the introduction nowadays we know well the
structure of algebraic sets and coordinate groups over a free
group $F$ of finite rank: irreducible coordinate groups over $F$ are
finitely generated limit groups over $F$. In our view there is
 sufficient background for studying algebraic geometry for
other classical algebraic structures.

\subsection{Free semigroup (monoid)}

Let $S$ be a free non-abelian semigroup or a free monoid.

There is well-known result due to G.\,Makanin that the problem of
solvability for system of equations over a free semigroup is
algorithmically solvable. Also there exist works on evaluation of
complexity for such algorithm. Nevertheless, we still do not know
a good description of sets of all solutions for system of
equations over $S$ (even for quadratic equations!) and their
coordinate semigroups.

Since, by Unification Theorem~A, the notion of coordinate algebra
is equivalent to the notion of limit algebra for equationally
Noetherian algebraic structures, then we formulate the following
problem.

\begin{problem}
It is interesting to find a description of limit semigroups over
free non-abelian semigroup $S$.
\end{problem}

\subsection{Free Lie algebra}

Let $L$ be the free Lie algebra of a finite rank $r\geqslant 2$ over
a field $k$.

\begin{problem}
It is interesting to develop general techniques for solving
equations over $L$ and develop the algebraic geometry over $L$.
\end{problem}

The guidelines for solution of this problem have been set up in
papers~\cite{BMR1,DMR1,MR2}. In the paper~\cite{Daniyarova3}
E.\,Daniyarova and V.\,Remeslennikov have produced results which
are specific for the free Lie algebra $L$ (see Example~\ref{Lie}
of current paper). So-called bounded algebraic geometry over free
Lie algebra $L$ has been completely examined
in~\cite{Daniyarova3}. It turns out that the algebraic geometry
over $L$ contains totally Diophantine algebraic geometry of the
ground field $k$.

It is well-known that quadratic equations, their algebraic sets
and coordinate groups have played a significative role in creation
of the algebraic geometry over the free group $F$. In our view in
the case of free Lie algebra linear equations may play a similar
role.

By ${\rm U}(L)$ we denote the universal enveloping algebra of $L$.
The algebra $L$ posses the natural structure of ${\rm
U}(L)$-module.

\begin{definition}
An equation of the form
$$
x_1\varphi_1+x_2\varphi_2+\ldots+x_n\varphi_n=w,
$$
where $w\in L$ and $\varphi_i\in {\rm U}(L)$, $i=\overline{1,n}$,
is called {\em linear} equation over $L$.
\end{definition}

Let us note that every expression $x\varphi$ ($\varphi\in {\rm
U}(L)$) may be written as a sum of terms in the form
$$
[\ldots[[x,v_1],v_2],\ldots v_m], \quad v_1,\ldots,v_n\in L.
$$

V.\,Remeslennikov and R.\,St\"{o}hr in their
paper~\cite{Rem-Shtor3} have demonstrated that the structure of
the solution of such simple an equation as $[x,a]+[y,b]=0$,
$a,b\in L$, $a\ne b$, is complicated. However, the coordinate
algebra for the equation $[x,a]+[y,b]=0$ may be calculated quite
easy.

\begin{problem}
It is interesting to develop a specific techniques for solving
linear equations over $L$, find corresponding algebraic sets and
coordinate algebras.
\end{problem}

\subsection{Free associative algebra}

Let $A$ be a free associative algebra of a finite rank $r\geqslant
2$ over a field $k$.

We know almost nothing about solutions of system of equations over
$A$. Thus we present the following ``testing problem'' for
realizing the algebraic geometry over $A$.

\begin{problem}
It is interesting to develop the bounded algebraic geometry over
$A$ in such a manner as it has been done over free Lie algebra
$L$.
\end{problem}

\subsection{Equationally Noetherian property}

Unification Theorems exhibit that the most perspective algebras
for investigation into algebraic geometry are equationally
Noetherian algebras. Thus we present the following open problems
about equationally Noetherian property for some classical
algebras.

\begin{problem}
Is the free non-abelian Lie algebra of a finite rank over a field
equationally Noetherian or not?
\end{problem}

\begin{problem}
Is the free non-abelian associative algebra of a finite rank over
a field equationally Noetherian or not?
\end{problem}

\begin{problem}
When the free product of equationally Noetherian groups is
equationally Noetherian?
\end{problem}

\bigskip

\bigskip

{\footnotesize \textsc{Evelina Daniyarova, Institute of
Mathematics
SB RAS, }\\
644099, Pevtsova 13, Omsk, Russia,\\
GSM: +79136304852\\
\texttt{evelina.omsk@list.ru}

\bigskip

\textsc{Alexei Myasnikov, Schaefer School of Engineering and
Science,
Department of Mathematical Sciences, Stevens Institute of Technology}\\
Castle Point on Hudson, Hoboken NJ 07030-5991, USA\\
+12012168598\\
\texttt{amiasnikov@gmail.com}

\bigskip

\textsc{Vladimir Remeslennikov, Institute of Mathematics SB RAS,}\\
644099, Pevtsova 13, Omsk, Russia,\\
GSM: +79136173632\\
\texttt{remesl@ofim.oscsbras.ru}
\end{document}